\documentclass{amsart}

\usepackage{}

\usepackage{graphicx}

\usepackage{url}
\usepackage{verbatim}

\usepackage{amsmath, amsfonts, amssymb, amsthm, float}

\usepackage{tikz}
\usetikzlibrary{positioning}
\usepackage{booktabs,subfig}   

\newcommand{\sfrac}[2]{#1/#2}

\newcommand{\Auti}{\operatorname{Aut}}

\newcommand{\PGU}{\operatorname{PGU}}

\def\Sz{{\rm{Sz}}}

\newcommand{\fqq}{\mathbb{F}_q}

\newcommand{\Gtwo}{G_{2}}

\newcommand{\GG}{^{2}G_{2}}
\newcommand{\AAA}{^{2}A_{2}}
\newcommand{\BB}{^{2}B_{2}}

\newcommand{\fgq}{\mathbb{F}_{q}}

\newcommand{\fg}[1]{\mathbb{F}_{#1}}
\newcommand{\fgqc}{\overline{\mathbb{F}_{q}}}

\newcommand{\JJr}{\mathcal{J}_{\text{R}}}

\newcommand{\NN}{\mathbb{N}}
\newcommand{\ZZ}{\mathbb{Z}}
\newcommand{\LL}{\mathcal{L}}

\newcommand{\xh}{X_{\text{H}}}
\newcommand{\xs}{X_{\text{S}}}
\newcommand{\xr}{X_{\text{R}}}
\newcommand{\fh}{F_{\text{H}}}
\newcommand{\fs}{F_{\text{S}}}
\newcommand{\fr}{F_{\text{R}}}
\newcommand{\gh}{g_\text{H}}
\newcommand{\gs}{g_\text{S}}
\newcommand{\gr}{g_\text{R}}

\newcommand{\Wedge}{\land^2}
\newcommand{\OO}{\mathcal{O}}
\newcommand{\im}{\text{Im}}

\newcommand{\pp}{\mathbb{P}}
\newcommand{\PP}{\mathbb{P}}
\newcommand{\ppfgqc}{\mathbb{P}^{13}(\fgqc)}
\newcommand{\ppfr}{\mathbb{P}_{\fr}}
\newcommand{\ppfx}{\mathbb{P}_{\fgq(x)}}

\newcommand{\Da}{\mathcal{D}}
\newcommand{\psii}{\psi_{\alpha \beta \gamma \delta}}

\newcommand{\gl}{\text{GL}(n,k)}

\newcommand{\fieldextension}{/}
\newcommand{\placeextension}{|}

\newtheorem{theorem}{Theorem}[section]
\newtheorem{lemma}[theorem]{Lemma}
\newtheorem{prop}[theorem]{Proposition}
\newtheorem{problem}[theorem]{Problem}
\newtheorem{cor}[theorem]{Corollary}

\theoremstyle{definition}

\theoremstyle{remark}
\newtheorem{remark}[theorem]{Remark}

\numberwithin{equation}{section}

\begin{document}

\title{Smooth Embeddings for the Suzuki and Ree Curves}


\author{Abdulla Eid}
\address{}
\curraddr{}
\email{eid1@illinois.edu}
\thanks{}

\author{Iwan Duursma}
\address{}
\curraddr{}
\email{duursma@illinois.edu}
\thanks{}


\date{}

\begin{abstract}
The Hermitian, Suzuki and Ree curves form three special families of curves
with unique properties. They arise as the Deligne-Lusztig varieties of
dimension one and their automorphism groups are the algebraic groups of
type 2A2, 2B2 and 2G2, respectively. For the Hermitian and Suzuki curves
very ample divisors are known that yield smooth projective embeddings of
the curves. In this paper we establish a very ample divisor for the Ree
curves. Moreover, for all three families of curves we find a symmetric set
of equations for a smooth projective model, in dimensions 2, 4 and 13,
respectively. Using the smooth model we determine the unknown nongaps in
the Weierstrass semigroup for a rational point on the Ree curve.
\end{abstract}

\maketitle


\section {Introduction}\label{chapter:1}



In this paper we study three important examples of algebraic curves over finite fields, the Hermitian, Suzuki, and Ree curves. These curves have many applications to algebraic geometry codes \cite{Duursma1}, \cite{DPark}, \cite{EidH}, \cite{Goppa}, \cite{Sti}, to exponential sums over finite fields \cite{Moreno}, and in finite geometry \cite{Hirschfeld}. The three curves are optimal curves with respect to Serre's explicit formula method. Their number of $\fgq$-rational points coincides with $N_{q}(g)$ which is the maximum number of $\fgq$-rational points for a curve of genus $g$ over $\fgq$. Moreover, they can be described (abstractly) as Deligne-Lusztig curves associated to the simple groups $\AAA$, $\BB$, and $\GG$, respectively \cite{Han}. The latter property suggests that we could apply the techniques that have been used for the Hermitian and Suzuki curves to find a very ample linear series for the Ree curve, construct smooth embeddings for the three Deligne-Lusztig curves above, and compute the Weierstrass non-gaps semigroup at the point at infinity for the Ree curve over $\fg{27}$. In this paper we also provide a complete set of five equations that define the Suzuki curve in the projective space $\mathbb{P}^{4}$ and 105 equations that define the Ree curve in the projective space $\mathbb{P}^{13}$. Moreover, these equations can be read easily from a complete graph with four and seven vertices, respectively.

The outline of this paper is as follows. In Section \ref{sec1:2} we introduce the three groups $\AAA$, $\BB$, and $\GG$ and the Deligne-Lusztig varieties. In Section \ref{sec1:3} we study the smooth embeddings for the Hermitian and Suzuki curves. In Section \ref{sec1:4} we generalize the techniques of Section \ref{sec1:3} to provide 105 equations that will be used in Section \ref{sec1:5} to construct a smooth model for the Ree curve. In Section \ref{sec1:6} we relate our work with the smooth embedding of the Ree curve as Deligne-Lusztig curve in \cite{Kane}. In Section \ref{sec1:7} we show that the Ree group acts on the smooth model in the projective space and finally in Section \ref{sec1:8} we compute the Weierstrass non-gaps semigroup at $P_\infty$ over $\fg{27}$.

\section{Preliminaries}\label{sec1:2}

\subsection{The groups $\AAA,\BB,\GG$}\label{sect2.1}

In this section we will recall the construction of the twisted groups $\AAA$, $\BB$, $\GG$. We begin our discussion with a historical background about the groups $A_2$, $B_2$, and $G_2$. Recall that these groups are the Lie groups associated to the Lie algebras of dimension 2 with the Dynkin diagrams and root systems as in Figures \ref{fig1:Dynkin} and  \ref{table1:RootSystem}.

In particular, the group $A_2(q)$ is the projective linear group $\text{PGL}(3,q)$ and the twisted group $\AAA(q)$ is the projective unitary group $\text{PGU}(3,q)$. 


The group $B_2(q)$, for $q=2^{2m+1}$, is the symplectic group $\text{PSp}_4(q)$. The group acts on a $3$-dimensional projective space $E$ and consists of those linear transformations that leave invariant a quadratic form $x_0 y_3 + x_1 y_2 + x_2 y_1 + x_3 y_0$. The choice of form is irrelevant but with hindsight on the consequences for the coordinates of the Suzuki curve we choose this form. Two mutually orthogonal vectors $(x_0,x_1,x_2,x_3)$ and $(y_0,y_1,y_2,y_3)$ span an isotropic line in the symplectic geometry. The symplectic group acts on the set $L$ of isotropic lines in $E$. A line in $E$ has Pl\"ucker coordinates $p_{ij} = x_i y_j + x_j y_i$ and the isotropic lines are precisely the lines with $p_{0,3} = p_{1,2}$. 

The Suzuki groups $\BB(q)=\text{Suz}(q)$ were originally defined as twisted Chevalley groups \cite{SuzNew}, \cite{CheNew}. Ree \cite{ReeNew} defines the group as the set of fixed points of the symplectic group under an involution. In the description of the group by Tits \cite{Tits} the involution arises through a polarity on the geometry of isotropic lines. We describe in some detail this polarity as it will explain the symmetry in the equations for the Suzuki curve.
 
A line with Pl\"ucker coordinates $p_{i,j}$ is incident with the point $(x_0,x_1,x_2,x_3)$ in $E$ if and only if
\begin{equation} \label{eq:inc}
\left [\begin {array}{cccc} 
p_{21}&p_{02}&p_{10}&0 \\
p_{13}&p_{30}&0&p_{01}\\
p_{32}&0&p_{03}&p_{20} \\
0&p_{23}&p_{31}&p_{12} 
\end {array}
\right ] 
\left [\begin {array}{c} 
x_0 \\
x_1 \\
x_2 \\
x_3
\end {array}
\right ]
= 0
\end{equation}
Let 
\[
X = \left [\begin {array}{cc} 
x_0 &x_2 \\ x_1 &x_3 
\end {array}
\right ],  \qquad
P = \left [\begin {array}{cc} 
p_{23} &p_{02} \\ p_{13} &p_{01} 
\end {array}
\right ]
\]
and let $\det X = d_X^2$ and $\det P = d_P^2$. For an isotropic line with $p_{03}=p_{12}=d_P$ we write (\ref{eq:inc}) as
\[
\left [\begin {array}{cc} 
d_P &0 \\  0 &d_P  
\end {array}
\right ] X =  \left [\begin {array}{cc} 
0 &p_{02} \\ p_{13} &0  
\end {array}
\right ] X + 
X 
\left [\begin {array}{cc} 
0 &p_{23} \\ p_{01} &0  
\end {array}
\right ].
\]
After multiplication on the left or on the right with $\text{adj}(X)$ and a comparison of the off-diagonal entries we obtain
\[
\left [\begin {array}{cc} 
 d_X^2 &0 \\  0 & d_X^2  
\end {array}
\right ] P =  \left [\begin {array}{cc} 
0 &x_2^2 \\ x_1^2 &0  
\end {array}
\right ] P + 
P 
\left [\begin {array}{cc} 
0 &x_0^2 \\ x_3^2 &0  
\end {array}
\right ]
\]
We can now formulate a duality for the symplectic geometry $(E,L)$ of projective $3$-space $E$ and its set of isotropic lines $L$.
Let $(F,M)$ denote another copy of the same geometry. For matrices $X$ and $P$, let $X \in E$ be a point and $(P,d_P) \in L$ be an isotropic line, and let
$P \in F$ be a point and $(X^{(2)},d_X^2) \in M$ be an isotropic line. Then
\[
\text{$X$ is incident with $(P,d_P)$ if and only if $P$ is incident with $(X^{(2)},d_X^2)$.}
\]


Now we give an explicit description of the Suzuki group $\BB(q)=\text{Suz}(q)$ ($q:=2^{2m+1}$, $m \in \NN$). Following Tits \cite{Tits}, let $E(x_0,x_1,x_2,x_3)$ be the 3-dimensional projective space with homogeneous coordinates $x_0,x_1,x_2,x_3$ over $\fgq$. Let $p_{ij}$ be the Pl\"ucker coordinates of lines in $E$ which clearly satisfy the relation $p_{01}p_{23}+p_{02}p_{13}+p_{03}p_{12}=0$. Consider the set $L$ of lines such that $p_{01}=p_{23}$. Let $V$ be the variety (hyperquadric) representing the set $L$ in the projective space $D(p_{01},p_{02},p_{03},p_{12},p_{31})$ which is given by the equation 
\[
p_{01}^2+p_{02}p_{31}+p_{03}p_{12}=0
\]
The automorphism group that leaves $E$ and $L$ invariant is $G(E,L):= B_2(q)$.

The tangent hyperplanes of $V$ intersect in a point with coordinates $p_{02}=p_{03}=p_{12}=p_{31}=0$, which allow us to inject $V$ into the 3-dimensional space $F(y_0,y_1,y_2,y_3)$ by $y_0:=p_{02}$, $y_1:=p_{31}$, $y_2:=p_{03}$, $y_3:=p_{12}$. Let $q_{ij}$ be the Pl\"ucker coordinates of lines in $F$. Now a line in $L$ passing through a point $(x_0,x_1,x_2,x_3)$ in $E$ will be represented in $F$ through $V$ as follows:
\begin{align*}
q_{01}=q_{23}=x_0x_1+x_2x_3, \quad q_{02}=x_0^2,\quad q_{03}=x_2^2\\
q_{31}=x_1^2 \quad q_{12}=x_3^2,
\end{align*}
which forms a set of lines $M$ in $F$ with the equation $q_{01}=q_{23}$. 

As a conclusion of the discussion above, we establish a duality between $(E,L)$ and $(F,M)$. This will give two maps $\delta:L \to F$ and a dual map $\delta':E \to M$ such that if a point $x \in E$ and a line $d \in L$, we have $x \in d$ and $\delta(d) \in \delta'(x)$ are equivalent.

On the level of groups, the duality above induces a group monomorphism $\delta^*:G(E,L) \to G(F,M)$ such that $\delta^*(g)(\delta'(d))=\delta'(g\cdot d)$, for $g \in G(E,L)$ and $d \in L$.

Next we introduce the polarity map which will give the definition of the twisted group $\BB$ and the set of $q^2+1$ $\fqq$-rational points. Next, let $\sigma$ be the automorphism group with $x^\sigma:=x^{2q_0}$ and consider the polarity map
\begin{align*}
f: E &\to F\\
(x_{0},x_{1},x_{2},x_{3})&\mapsto (y_0=x_{0}^\sigma,y_1=x_{1}^\sigma,y_2=x_{2}^\sigma,y_3=x_3^\sigma).
\end{align*}

Note that if $x,x' \in E$, then the relation "$f(x)$ belongs to $\delta'(x')$" is symmetric. Thus, the map $f$ appears like a polarity map. Let $\Gamma$ be the set of points in $x \in E$ such that $f(x) \in \delta'(x)$.

The map $f$ defines a group homormorphism $f^*:G(E,L) \to G(F,M)$ by $f^*(g)(f(x)):=f(g\cdot x)$, where $x\in E$, $g \in G(E,L)$. Define
\begin{align*}
G^*(f)&:=\{ g \in G(E,L) \,:\, f^*(g)=\delta^*(g) \},\\
G^* &:= \{ g \in G(E,L) \,:\, g\cdot x =x, \,\forall x \in \Gamma \}, 
\end{align*}
i.e., $G^*(f)$ is the group of all automorphisms that leave $f$ invariant which is a subgroup of $G^*$. As in \cite{Tits}, $G^*(f)=G^*$ and the Suzuki group $\BB(q)$ is defined to be $G^*$. We note that $G^*$ acts transitively on the set $\Gamma$ which is the set of $\fgq$-rational points of size $q^2+1$. Moreover, $G^{*}$ acts on $E$. Hence, it acts on 3-dimensional projective space.

We end the discussion of the Suzuki group by stating the defining equation of the set $\Gamma$. Let $x,z \in \fqq$, let $\gamma(x,z):=[1:x:z:w] \in \mathbb{P}^{3}$, where
\[
w:=xz + x^{\sigma+2}+z^{\sigma}
\]
Then, the set of all $\fqq$-rational points can be described as the set 
\[
\Gamma:=\{ \gamma(x,z) \in \mathbb{P}^3  \,|\, x,z \in \fqq  \} \cup \{[0:0:0:1] \}
\]
The Suzuki group is the group that leaves $\Gamma$ above invariant, i.e., $\BB(q) \subseteq  \text{PGL}(4,\fgq)$. Moreover, the Suzuki group $\BB(q)$ acts 2-transitively on the set $\Gamma$. Using the above, we note that the set $\Gamma$ is the set of $\fgq$-rational points that corresponds to the set of rational places of the function field $F':=\fgq(x,z)$ defined by
\begin{align}
 z^q-z&=x^{2q_0} (x^q-x) \label{eq:w1x},
\end{align}
which has $q^2$ affine $\fgq$-rational places and one place at infinity. The full description of how the automorphism group $\BB$ acts can be found in \cite[Section 4.3]{Tits}.



Now we discuss the groups $G_2(q)$ and $\GG(q)$ as it appeared in Dickson \cite{Dickson} and Tits \cite{Tits} respectively. In \cite[Section 9]{Dickson}, Dickson originally described the group $G_2(q)$ as the group of linear homogeneous transformations on the seven variables $\xi_0,\xi_1,\xi_2,\xi_3,\mu_1,\mu_2,\mu_3$ over the field $\fgq$ which leaves the equation
\[
 \xi_0^2 + \xi_1\mu_1 + \xi_2\mu_2 +\xi_3\mu_3
\]
and the system of equations
\begin{equation}\label{eq:*}
\begin{aligned}
X_1 + Y_{23}=0 , \quad &X_2+Y_{31}=0, \quad X_3 + Y_{12}=0, \\
Y_1+X_{23}=0, \quad &Y_2 + X_{31}=0, \quad Y_3+X_{12}=0
\end{aligned}
\end{equation}

invariant, where
\begin{alignat}{2}
X_i:=&\left|\begin{array}{cc}\xi_0&\xi_i\\ \overline{\xi _0}& \overline{\xi _i} \end{array}\right|,
\,Y_i:=&&\left|\begin{array}{cc}\xi_0&\mu_i\\ \overline{\xi _0}& \overline{\mu _i} \end{array}\right|,\\
X_{ij}:=&\left|\begin{array}{cc}\xi_i&\xi_j\\ \overline{\xi_i}& \overline{\xi_j} \end{array}\right|,
\, Y_{ij}:=&&\left|\begin{array}{cc}\mu_i&\mu_j\\ \overline{\mu_i}& \overline{\mu_j} \end{array}\right|,
\, Z_{ij}:=\left|\begin{array}{cc}\xi_i&\mu_j\\ \overline{\xi_i}& \overline{\mu_j} \end{array}\right|,
\end{alignat}
and $\overline{\xi_i}, \overline{\mu_i}$ are the conjugate of $\xi_i,\mu_i$, i.e., $\overline{\xi_i}=\xi^q, \overline{\mu_i}=\mu_i^q$.


Since the Ree group $\GG(q)\subseteq G_2(q)$ is a subgroup, we expect that we can add more equations to obtain a subvariety with an action of $\GG(q)$. 

Next we give an explicit construction of the Ree group $\GG(q)$ ($q:=3q_0^2$, $q_0:=2^m$, $m \in \NN$). Tits \cite{Tits} carried Dickson's idea further to the Ree group $\GG(q)$ and showed that $\GG(q)$ acts on a seven dimensional space. Following the notations in \cite{Tits}, let $P(x_*,x_0,x_1,x_2,x_{0'},x_{1'},x_{2'})$ be the 6-dimensional projective space $P$ with homogeneous coordinates $x_*,\dots,x_{2'}$ over $\fgq$ (where all arithmetic on the indices is modulo 3). Let $E$ be the quadric defined by the equation $x_*^2+\sum_{i=0}^{2}x_ix_{i'}=0$ and let $L$ be the set of lines defined by
\begin{equation}\label{eq:Tits6}
\left.  \begin{aligned}
p_{*i}+p_{(i+1)'(i+2)'}=0,\\
p_{*i'}+p_{(i+1)(i+2)}=0,\\
\sum_{i=0}^{2}p_{ii'}=0,\\
\end{aligned}
\quad \right \}
 \qquad \text{Similar to } \eqref{eq:*}
\end{equation}
where $p_{i,j}$ is the Pl\"ucker coordinate as before. 
Therefore, the automorphism group that leaves $E$ and $L$ invariant is $G(E,L):=G_2(q)$. Now let $V$ be a 5-dimensional variety representing $L$ in the projective space $D(p_{*i},p_{*i'},p_{ij'}\, : \,\sum p_{ii}=0)$ of dimension 13 over $\fgq$. Define the 6-dimensional projective space $Q(y_{*},y_{i},y_{i'})$ by $y_{*}:=p_{00'}-p_{11'}$, $y_{i}:=p_{i+1,i'}$, and $y_{i'}:=p_{i(i+1)'}$. Then, $V$ will be mapped into the quadric $F: y_{*}^2+\sum_{i=0}^{2}y_iy_{i'}=0$ in $Q(y_{*},y_{i},y_{i'})$ (which might be singular, see \cite{Tits}).

Let $(x_{*},x_{i},x_{i'}) \in E$ be a point in $E$, we define the set of lines $M$ in $F$ using the equations \eqref{eq:Tits6} by replacing the $p_i$'s with $q_i$'s, the Pl\"ucker coordinates of $Q$, such that $x_{*}^3=q_{00'}-q_{11'}$, $x_i^3=q_{(i+1)i'}$, and $x_{i'}^3=q_{i(i+1)'}$. This will give two maps $\delta:L \to F$ and $\delta':E \to M$. Therefore, we will have a group homomorphism $\delta^*:G(E,L) \to G(F,M)$ such that $\delta^*(g)(\delta'(d))=\delta'(g\cdot d)$, where $g \in G(E,L), d \in L$.

Next, let $\sigma$ be the automorphism with $x^\sigma:=x^{3q_0}$ and consider the polarity map
\begin{align*}
f: E &\to F\\
(x_{*},x_{i},x_{i'})&\mapsto (x_{*}^\sigma,x_{i}^\sigma,x_{i'}^\sigma).
\end{align*}
This map defines a group homormorphism $f^*:G(E,L) \to G(F,M)$ by $f^*(g)(f(x)):=f(g\cdot x)$, where $x\in E$, $g \in G(E,L)$. Let $\Gamma$ be the set of points in $E$ such that $f(x) \in \delta'(x)$. Define
\begin{align*}
G^*(f)&:=\{ g \in G(E,L) \,:\, f^*(g)=\delta^*(g) \},\\
G^* &:= \{ g \in G(E,L) \,:\, g\cdot x =x, \,\forall x \in \Gamma \}, 
\end{align*}
i.e., $G^*(f)$ is the group of all automorphisms that leave $f$ invariant which is a subgroup of $G^*$. As in \cite{Tits}, $G^*(f)=G^*$ and the Ree group $\GG$ is defined to be $G^*$. We note that $G^*$ acts transitively on the set $\Gamma$ which is the set of $\fgq$-rational points of size $q^3+1$. Moreover, $G^{*}$ acts on $E$. Hence, it acts on 6-dimensional projective space.

We give different interpretation of Condition \eqref{eq:Tits6}. Let $E$ be the variety with points the nonzero $3 \times 3$-matrices 
\begin{equation}\label{eq:point} 
\left[ \begin{array}{ccc} 
0 &x_1 &x_{-2} \\ x_{-1} &-x_0 &x_3 \\ x_2 &x_{-3} &x_0 
\end{array} \right], \qquad x_0^2+x_1 x_{-1} + x_2 x_{-2} + x_3 x_{-3} = 0.
\end{equation}
With this condition the characteristic polynomial of a matrix $X \in E$ is of the form $T^3 - \det X = 0$ and the matrix has a unique eigenvalue $d_X$. For two matrices $X$ and $Y$ in $E$, the linear span of $X$ and $Y$ forms a line $L$ in $E$ in the sense of \eqref{eq:Tits6} if and only if
\begin{equation} \label{eq:line}
[X,Y] = XY-YX \in \langle I \rangle  
\end{equation}
The Conditions \eqref{eq:Tits6} or \eqref{eq:line} take the form
\[
 \left[ \begin {array}{ccccccc} 
0&-y_1&-y_2&-y_3&y_{-1}&y_{-2}&y_{-3} \\ 
y_1&0&-y_{-3}&y_{-2}&y_0&0&0 \\ 
y_2&y_{-3}&0&-y_{-1}&0&y_0&0 \\ 
y_3&-y_{-2}&y_{-1}&0&0&0&y_0\\
-y_{-1}&-y_0&0&0&0&y_3&-y_2 \\
-y_{-2}&0&-y_0&0&-y_3&0&y_1\\
-y_{-3}&0&0&-y_0&y_2&-y_1&0
\end {array} \right] 
\left[ \begin {array}{c} -x_0  \\ x_{-1} \\ x_{-2}  \\ x_{-3} \\ x_1 \\ x_2 \\ x_3 
\end {array} \right] 
= 0.
\]
This can be written in a short form as
\begin{align*}
&x^+ \cdot y^- = y^+ \cdot x^-  ( = -x_0 y_0 ) \\
&x^+ \times y^+ = - x_0 y^- + y_0 x^- \\
&x^-\times y^- = + x_0 y^+ - y_0 x^+ ,
\end{align*}
where $x^+:=(x_1,x_2,x_3)$, $x^-:=(x_{-1},x_{-2},x_{-3})$. Note
\[
x^+ \cdot x^- = - x_0^2, \qquad  y^+ \cdot y^-  = -y_0^2.
\]
Thus, we have the two $2 \times 4$ orthogonal spaces
\[
\langle (x_0, x^+), (y_0, y^+) \rangle ~\perp~ \langle (x_0,x^-), (y_0,y^-) \rangle.
\] 
For the line $L$ through $X$ and $Y$, let 
\begin{align}
P^\ast 
&~=~ 
\left[ \begin{array}{ccc} 
p_{1,-1} &p_{1,-2} &p_{1,-3} \\
p_{2,-1} &p_{2,-2} &p_{2,-3} \\
p_{3,-1} &p_{3,-2} &p_{3,-3}
\end{array} \right]
\label{eq:P} \\
&~=~
\left[ \begin{array}{c} 
x_1 \\
x_2 \\
x_3 
\end{array} \right]
\left[ \begin{array}{ccc} 
y_{-1} &y_{-2} &y_{-3} 
\end{array} \right]
~-~
\left[ \begin{array}{c} 
y_1 \\
y_2 \\
y_3 
\end{array} \right]
\left[ \begin{array}{ccc} 
x_{-1} &x_{-2} &x_{-3} 
\end{array} \right] \nonumber
\end{align}
The matrix $P^\ast$ has trace $p_{1,-1}+p_{2,-2}+p_{3,-3} = (x^+ \cdot y^-) - (y^+ \cdot x^-) = 0.$ After subtracting $p_{1,-1} I$ the matrix $P = P^\ast - p_{1,-1} I \in E$. The lines in $E$ through $X$ form a pencil. Let $L(X,Y)$ and $L(X,Y')$ be two independent lines in this pencil, with matrices $P$ and $P'$, respectively. The two matrices $P$ and $P'$ span a line in $E$ in the sense of \eqref{eq:line}. Let $Q$ be the matrix defined by \eqref{eq:P} for the line through $P$ and $P'$. Then $Q = X^{(3)}$, which can be summarized as
\begin{equation} \label{eq:polree}
P(P(X,Y),P(X,Y')) = X^{(3)}.
\end{equation}

We conclude this section by stating Cohen's construction \cite{Cohen}. Cohen used Tits' construction above to define the Ree group as the automorphism group acting on the set $\Gamma$. More specifically, for $x,y,z\in \fgq$, let $\gamma(x,y,z):=[x:y:z:1:u:v:w] \in \mathbb{P}^{6}$ such that $u,v$, and $w$ are defined by the equations
\begin{align*} \label{eq:uvw}
 &u=x^2y-xz+y^\sigma - x^{\sigma+3},\\
 &v=x^\sigma y^\sigma -z^\sigma+xy^2+yz-x^{2\sigma+3},\\
 &w=-z^2-xv-yu.
\end{align*}
Then, the set of all $\fgq$-rational points can be defined as the set
\[
 \Gamma:=\{ \gamma(x,y,z) \in \mathbb{P}^6 \,|\, x,y,z \in \fgq \} \cup \{ [0:0:0:0:0:0:1] \}.
\]
Cohen defined the Ree group as the group of all projective linear transformations leaving $\Gamma$ invariant, i.e., $\GG(q)\subseteq \text{PGL}(7,\fgq)$. Moreover, the Ree group $\GG(q)$ acts 2-transitively on the set $\Gamma$. Using the above, we note that the set $\Gamma$ is the set of $\fgq$-rational points that corresponds to the set of rational places of the function field $F':=\fgq(x,y,z)$ defined by
\begin{align}
 y^q-y&=x^{3q_0} (x^q-x) \label{eq:w1x},\\
 z^q-z&=(x^{q_0+1}-y^{q_0})(x^q-x)\label{eq:w2xw1}
\end{align}
which has $q^3$ affine $\fgq$-rational places and one place at infinity. The full description of how the automorphism group $\GG$ acts can be found in Section \ref{sec1:7}.


\subsection{The Deligne-Lusztig Curves}\label{sec1:2.2}


The Hermitian, Suzuki, and Ree curves can be described (abstractly) as Deligne-Lusztig curves associated to the simple groups $\AAA$, $\BB$, and $\GG$, respectively. In this subsection we introduce their structure as Deligne-Lusztig curves and we list their basic properties. We refer the reader to the books \cite{Cart1},\cite{Cart2},\cite{humphreys},\cite{Hurt},\cite{Malle},\cite{springer} and the papers \cite{DL},\cite{Han} for a full treatment of the subject. Here we will follow the notations and the exposition in \cite{Han}.

Let $G$ be a connected algebraic group over a finite field $k=\fgq$, i.e., $G$ is an affine variety defined over $k$ such that $G$ is also a group in which both the multiplication and inversion maps are morphisms and $G$ is a connected topological space in the Zariski topology on $G$. Thus, $G$ can be regarded as a closed subgroup of the linear group $\gl$, for some $n\in \NN_{>0}$. The \emph{unipotent radical} of $G$ is the maximal closed connected normal subgroup all of whose elements are \emph{unipotent elements}\footnote{An element $a$ in a ring $A$ is called a \emph{unipotent element} if $1-a$ is a nilpotent element.}. $G$ is called \emph{reductive} if the unipotent radical of $G$ is trivial. A subgroup $B\subseteq G$ is called a \emph{Borel} subgroup if it is a maximal connected solvable subgroup. An algebraic group is called a \emph{torus} if it is isomorphic to $k^\times \times k^\times \times \cdots \times k^\times$ (as embedded in $\gl$).

For our purpose, let $G$ be a connected reductive algebraic group embedded in $\gl$ with a Borel subgroup $B$ and a maximal torus $T$ \footnote{$G$ will be $A_2$, $B_2$, or $G_2$ for our purpose.}. The \emph{Weyl} group of $G$ is the finite group $W:=W(G):=N_{G}(T)/T$, where $N_G(T)$ is the normalizer of $T$ in $G$. We note here that the Weyl group is a Coxeter group which means it is generated by a set of generators $s_i$'s, called the \emph{reflections}, of order 2 and has a presentation with relations of the form $(s_is_j)^{m_{ij}}$ ($m_{ij}=2,3,4,6$), i.e.,
\[
 W(G):= \left \langle s_1,\dots,s_{r} \mid (s_i)^2=1,\, (s_is_j)^{m_{ij}}=1\, \text{ for } i \neq j  \right \rangle.
\]

Let
\begin{align*}
 \text{Fr}_{q}: G \subseteq \gl &\to G \subseteq \gl \\
                (a_{ij}) &\mapsto (a_{ij}^{q})
\end{align*}
be the standard Frobenius map which will define a map $\sigma:G \to G$ such that some power of $\sigma$ is the standard Frobenius map\footnote{$\sigma^2=\text{Fr}_q$ , i.e., $\sigma(g)=g^{\sqrt{q}}$ for $G=A_2$, $B_2$, and $G_2$, respectively.} $\text{Fr}_q$ \cite[Page 183]{Malle}. $\sigma$ is called the \emph{Frobenius map}\footnote{In some literature it is called the \emph{Steinberg automorphism} \cite[Page 183]{Malle}.}. Denote the fixed group of the Frobenius map $\sigma$ by $G^{\sigma}$ \footnote{$G^\sigma=\,\AAA$, $\BB$, and $\GG$ if $G=A_2$, $B_2$, and $G_2$, respectively.} . It is called a \emph{finite group of Lie type} \cite[Theorem 21.5]{Malle}.

Let $X_G:=\{B \subseteq G \,|\, B \text{ is a Borel subgroup of } G \}$. Since any two Borel subgroups of $G$ are conjugate by an element in $G$, we have that $G$ acts transitively on $X_G$ by conjugation. Moreover, using the Lang-Steinberg Theorem \cite[Theorem 21.7]{Malle}, which asserts that the map $L:G \ni g \mapsto g^{-1}\sigma(g)\in G$ is surjective, we have that any two $\sigma$-stable Borel subgroups are conjugate by an element in $G^\sigma$ \cite[Section 2.2.2]{Han}.
Therefore, the group $G^\sigma$ acts on the set of $\sigma$-stable Borel subgroups by conjugation. We also have a natural bijection $G/B \ni gB \mapsto gBg^{-1} \in X_G$, for a fixed Borel subgroup $B$.

In \cite{Han}, we can identify the set of orbits of $G$ in $X_{G} \times X_{G}$ with the Weyl group $W$. For $w\in W$, the orbit in $X_G \times X_G$ corresponding to $w$, denoted by $\mathcal{O}(w)$, is given by
\begin{align*}
 \mathcal{O}(w):&=\{(g_1B,g_2B) \in G/B \times G/B \,|\, g_1^{-1}g_2\in BwB \}.
\end{align*}
We say that two Borel subgroups $B_1,B_2$ of $G$ are in \emph{relative position} $w$ if $(B_1,B_2)\in \mathcal{O}(w)$. Define the Deligne-Lusztig variety $X(w)$ to be
\[
 X(w):=\{ B' \in X_G \,|\, (B',\sigma(B'))\in \mathcal{O}(w) \}.
\]
If we identify $X_G \simeq G/B$, for a fixed Borel subgroup $B$ of $G$, then as in \cite{Hurt} we have that
 \[
  X(w)=\{ (gB,\sigma(g)B) \in G/B \times G/B \,|\, g^{-1}\sigma(g) \in BwB \}.
 \]
We have the following proposition.
\begin{prop}\cite{DL,Han}\label{prop:2.1}
\quad
\begin{enumerate}
 \item $\dim(X(w))=\text{length}(w)=n$, where $w=s_1s_2\cdots s_n$ is a product of reflections.
 \item $X(w)$ is irreducible if and only if for every simple reflection $s \in W$, $s$ is in the $\sigma$-orbit of some $s_i$ ($i=1,2,\dots ,n$)
 \item $X(w)$ is $G^\sigma$-stable.
 \item $X(e)\subseteq X(w)$ is the set of all $\fgq$-rational points. 
 \item $G^\sigma$ acts on $X(e)$ and $G^\sigma=\Auti(X(w))$
 \item If $w=s$ is a simple reflection, then in particular we have the Deligne-Lusztig curve $\overline{X}(w):=X(w)\cup X(e)$ which is a curve with the group $G^\sigma$ acting as the $\fgq$-rational automorphism group and it is irreducible if and only if every simple reflection $s' \in W$ is in the $\sigma$-orbit of $s$.
\end{enumerate}
\end{prop}

\begin{remark}
 If we require $G^\sigma$ to be a simple group, then $G^\sigma$ is either $\AAA$, $\BB$ or $\GG$. In that case, the Weyl group $W(G)$ has two generators $s_1,s_2$ of order 2 with $(s_1s_2)^{m_{12}}=1$, where $m_{12}=3,4,6$ for $G=A_2$, $B_2$, and $G_2$, respectively. The Dynkin diagram has two vertices corresponding to the two simple positive roots $\alpha,\beta$ of $G$ as in Figure~\ref{fig1:Dynkin}.
\begin{figure}[htb!]
\centering
\subfloat[ $A_2$]
{
\begin{tikzpicture}
  [scale=1,auto=left,]
  \node [circle,draw] (a) at (0,0) {};
  \node [circle,draw] (b) at (2,0)  {};

  \draw [-] (a) -- (b);
   \node [] (1) at (-0.4,0)  {$\beta$};
  \node [] (1) at (2.4,0) {$\alpha$};

\end{tikzpicture}

}
\subfloat[ $B_2$]
{
\begin{tikzpicture}
  [scale=1,auto=left,]
  \node [circle,draw] (a) at (0,0) {};
  \node [circle,draw] (b) at (2,0)  {};

  \draw [->] (a.north) --  (b.north);
  \draw [->] (a.south) -- (b.south);

  \node [] (1) at (-0.4,0)  {$\beta$};
  \node [] (1) at (2.4,0) {$\alpha$};
\end{tikzpicture}
}
\quad \subfloat[ $G_2$]
{
\begin{tikzpicture}
  [scale=1,auto=left,]
  \node [circle,draw] (a) at (0,0) {};
  \node [circle,draw] (b) at (2,0)  {};

  \draw [->] (a) -- (b);
  \draw [->] (a.north) --  (b.north);
  \draw [->] (a.south) -- (b.south);
   \node [] (1) at (-0.4,0)  {$\beta$};
  \node [] (1) at (2.4,0) {$\alpha$};

\end{tikzpicture}
}
\caption{The Dynkin diagram for $G$.}\label{fig1:Dynkin}

\end{figure}
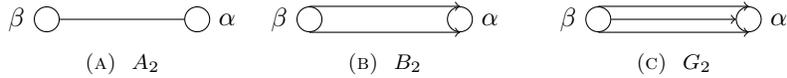

The Dynkin diagram for the group $G^\sigma$ is the same diagram as in Figure~\ref{fig1:Dynkin} with the $\sigma$-action permuting the two roots.
\end{remark}

\begin{remark}
 The two dimensional root system of the group $G=A_2$, $B_2$, or $G_2$ are shown in Figure~\ref{table1:RootSystem}, where $\alpha$ represents the short root and $\beta$ represents the long root \cite[Table A.2]{Malle}.
\begin{figure}
\centering
\subfloat[ $A_2$]
{
\begin{tikzpicture}
  [scale=1,auto=left, minimum size=5em]


  \draw [->] (0,0) --    (1,0);
  \draw [->] (0,0) --   (0.5,0.866);
  \draw [->] (0,0) --  (-0.5,0.866);

  \draw [->] (0,0) --  (-1,0);


 \draw [->] (0,0) --   (-0.5,-0.866);
  \draw [->] (0,0) --  (0.5,-0.866);

 \node [] (1) at (1.3,0) {$\alpha$};
 \node [] (1) at (0.8,1.166) {$\alpha+\beta$};
 \node [] (1) at (-0.8,1.166) {$\beta$};

\end{tikzpicture}
}
\subfloat[$B_2$]
{
\begin{tikzpicture}
  [scale=1,auto=left, minimum size=5em]


  \draw [->] (0,0) --    (1,0);
  \draw [->] (0,0) --   (1.06,1.06);
  \draw [->] (0,0) --  (0,1);
  \draw [->] (0,0) --  (-1.06,1.06);
  \draw [->] (0,0) --  (-1,0);


  \draw [->] (0,0) --   (-1.06,-1.06);
  \draw [->] (0,0) --  (0,-1);
  \draw [->] (0,0) --  (1.06,-1.06);


\node [] (1) at (1.3,0) {$\alpha$};
\node [] (1) at (1.36,1.36) {$2\alpha+\beta$};
\node [] (1) at (0,1.3) {$\alpha+\beta$};
\node [] (1) at (-1.36,1.36) {$\beta$};

\end{tikzpicture}
}
\subfloat[ $G_2$]
{
\begin{tikzpicture}
  [scale=1,auto=left, minimum size=5em]


  \draw [->] (0,0) --    (1,0);
  \draw [->] (0,0) --   (1.3,0.75);
  \draw [->] (0,0) --  (0.5,0.866);
  \draw [->] (0,0) --  (0,1.5);
  \draw [->] (0,0) --  (-0.5,0.866);
  \draw [->] (0,0) --  (-1.3,0.75);
  \draw [->] (0,0) --  (-1,0);


  \draw [->] (0,0) --   (-1.3,-0.75);
  \draw [->] (0,0) --  (-0.5,-0.866);
  \draw [->] (0,0) --  (0,-1.5);
  \draw [->] (0,0) --  (0.5,-0.866);
  \draw [->] (0,0) --  (1.3,-0.75);


\node [] (1) at (1.3,0) {$\alpha$};
\node [] (1) at (2,0.566) {$3\alpha+\beta$};
\node [] (1) at (0.8,1.166) {$2\alpha+\beta$};
\node [] (1) at (0,1.8) {$3\alpha+2\beta$};
\node [] (1) at (-0.8,1.166) {$\alpha+\beta$};
\node [] (1) at (-1.6,0.75) {$\beta$};

\end{tikzpicture}
}

\caption{The root system for $G$.}\label{table1:RootSystem}

\end{figure}
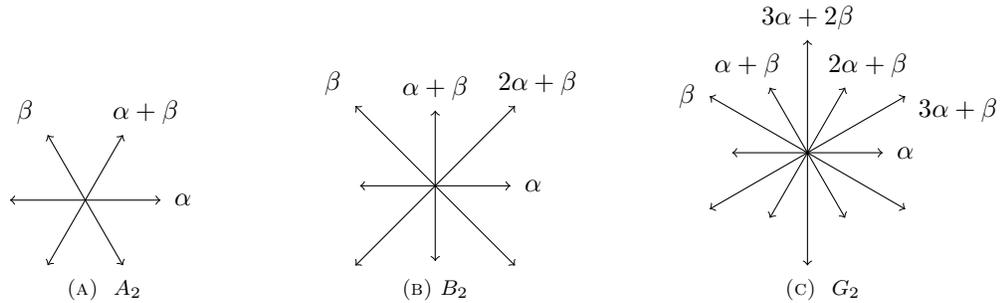
Note that a root system in Figure~\ref{table1:RootSystem} determines an underlying simple Lie algebra $\mathcal{G}$ with $G$ as its Lie group \cite{Agricola}.
\end{remark}

As above, let $G^\sigma$ be a simple group, i.e., $G^\sigma$ is either $\AAA$, $\BB$ or $\GG$. In Table \ref{table:1.1.1} we summarize some information of the Deligne-Lusztig curves such as the automorphism group, the number of $\fgq$-rational points, and the genus \cite{Han}.


\begin{table}[htb!]
\begin{center}
\begin{tabular}{ cccc  }
\toprule
Curve & Hermitian & Suzuki & Ree \\
\midrule
 $G$ & $A_2$& $B_2$ & $G_2$ \\
 $G^\sigma$ & $\AAA$& $\BB$ & $\GG$ \\
 $\left | G^\sigma \right | $ & $q_0^3(q-1)(q_0^3+1) $& $q^2(q-1)(q^2+1)$ & $q^3(q-1)(q^3+1)$ \\

$N_1$ & $q_0^3+1$& $q^2+1$ & $q^3+1$ \\

 $g$ & $\frac{1}{2}q_0(q_0-1)$ & $q_0(q-1)$ & $\frac{3}{2}q_0(q-1)(q+q_0+1)$ \\
\bottomrule
\end{tabular}

\end{center}
\caption{Information about the Deligne-Lusztig curves associated to the groups $\AAA$, $\BB$, and $\GG$.}
\label{table:1.1.1}
\end{table}

These three curves are realized as the projective curves corresponding to the following algebraic function fields which have the same number of $\fgq$-rational points, genus, and automorphism group as in Table \ref{table:1.1.1}, see \cite{HP},\cite{Tor3}.

\begin{enumerate}
 \item The Hermitian curve \cite{Sti} corresponds to $\fh:=\fgq(x,y)$ over $\fgq$ ($q:=q_0^2$, $q_0$ is a prime power) defined by the equation
 \[
  y^{q_0}+y=x^{q_0+1}.
 \]

 \item The Suzuki curve \cite{HS} corresponds to $\fs:=\fgq(x,y)$ over $\fgq$ ($q:=2q_0^2$, $q_0:=2^m$, and $m \in \NN$) defined by the equation
\[
 y^q-y=x^{q_0}(x^q-x).
\]

 \item The Ree curve \cite{HP},\cite{Ped} corresponds to $\fr:=\fgq(x,y_1,y_2)$ over $\fgq$ ($q:=3q_0^2$, $q_0:=3^m$, and $m \in \NN$) defined by the two equations
\begin{align*}
y_1^q-y_1=x^{q_0}(x^q-x),\\
y_2^q-y_2=x^{q_0}(y_1^q-y_1).
\end{align*}
\end{enumerate}

Kane \cite{Kane} used the construction of these three curves as Deligne-Lusztig curves to give smooth embeddings of these curves in the projective space of dimension 2, 4, and 13, respectively. We will use the function field description above to give a smooth embedding in projective space of dimensions 2, 4, and 13. In Section \ref{sec1:6} we will show that for the Ree curve, the set of $\fgq$-rational points is the same for our embedding and for Kane's embedding.


\section{The Smooth Embeddings for the Hermitian and Suzuki Curves}\label{sec1:3}
In this section we study the smooth embeddings for the two Deligne-Lusztig curves associated to the groups $\AAA$ and $\BB$. These curves are known as the Hermitian and Suzuki curves respectively. We will use the same techniques of this section to construct a smooth embedding for the third Deligne-Luszig curve associated to the group $\GG$ which is the Ree curve. This has also been done independently by Kane \cite{Kane}, where he provided a systematic approach to find smooth embeddings for these curves in the projective space in a uniform way. His approach was to use the structure of the curves as Deligne-Lusztig curves with Borel subgroups as the points on the curves. Our approach is to use the function field description of these curves as given by equations. In Section \ref{sec1:6} we will show that for the Ree curve, these two different embeddings give the same $\fgq$-rational points. The following work is motivated by the work of Tits \cite{Tits} who considered the idea of the line between a point and its Frobenius image and the use of Pl\"{u}cker coordinates. In particular, in the preliminary results for the Suzuki curve announced in \cite{DEmail},\cite{DTalk}, where the author gave five defining equations for the smooth model of the Suzuki curve using the idea of Pl\"ucker coordinates. We generalize that approach to a uniform approach for all three Deligne-Lustig curves.

We review first the smooth embedding for the Hermitian curve. 

\subsection{The Hermitian Curve }\label{sec1:3.1}
The Hermitian curve has been studied in detail in \cite[Chapter 6]{Sti}. 
It is given by the affine equation $y^{q_0}+y=x^{q_0+1}$ over $\fgq$ ($q:=q_0^2$, $q_0$ is a prime power). It has $q_0^3+1$ $\fgq$-rational points with one point at infinity $P_\infty$ and is of genus $\gh=\sfrac{q_0(q_0-1)}{2}$. Hence, the Hermitian curve attains the Hasse-Weil bound. Therefore, it is a maximal and optimal curve with $L$-polynomial $L(t):=(q_0t+1)^{2\gh}$. Moreover, the Hermitian curve is the unique curve of genus $\gh=\sfrac{q_0(q_0-1)}{2}$ and number of $\fgq$-rational points equals to $q_0^3+1$ \cite{StiRuck}. The automorphism group of the Hermitian curve is $\AAA=\PGU(3,q)$. Moreover, the equation $y^{q_0}+y=x^{q_0+1}$ defines a smooth model for the Hermitian curve.
\begin{remark}
Let $H:=(q_0+1)P_\infty$. Then, the linear series $\Da:=\left |H\right |$ is a very ample linear series of dimension 2 generated by $1,x,y$. Therefore, the morphism associated to $\Da$ is a smooth embedding for the Hermitian curve in $\mathbb{P}^{2}(\fgqc)$.
\end{remark}
\begin{remark}
 The tangent line at a point $P$ in the Hermitian curve is given by the equation
\[
 1_{P}^{q_0}\cdot y-x_{P}^{q_0}\cdot x+y_{P}^{q_0}\cdot 1=0.
\]
\end{remark}

\begin{remark}\label{ch1:hermitianremrk}
The Hermitian curve can also be defined in $\pp^2(\fgqc)$ using the equation \cite[Section 6.4]{Sti}
\[
v^{q_0+1}+u^{q_0+1}+1=0.
\]
\end{remark}

Now from the defining equation $y^{q_0}+y=x^{q_0+1}$ of the Hermitian curve, we get that
\begin{equation}\label{append:1}
\begin{pmatrix}
 1^{q_0} &  -x^{q_0} & y^{q_0}
\end{pmatrix} \begin{pmatrix}
y & y^q \\
x & x^q \\
1 & 1^q
\end{pmatrix} = 0.
\end{equation}
Consider the following matrix $H$
\[
H = \left( \begin{array}{lllllll}
1 &:~x &:~y\\ \medskip
1 &:~x^q &:~y^q
\end{array} \right).
\]
Let $H_{i,j}$ be the Pl\"ucker coordinates of the matrix $H$, i.e., $H_{1,2}=x^q-x$, $H_{3,1}=y-y^q$, and $H_{2,3}=xy^q-yx^q$. Then,
\begin{equation}\label{append:2}
\begin{pmatrix}
 H_{1,2}&  H_{3,1} &  H_{2,3}
\end{pmatrix} \begin{pmatrix}
y & y^q \\
x & x^q \\
1 & 1^q
\end{pmatrix} = 0.
\end{equation}
Note that Equations \eqref{append:1} and \eqref{append:2} define two lines between a point $P:=(1,x,y)$ and its Frobenius image $P^{(q)}:=(1,x^q,y^q)$. But the line between a point and its Frobenius image is unique (in fact, it is the tangent line at $P$). Therefore, $1^{q_0}$ is proportional to $H_{1,2}$, $-x^{q_0}$ is proportional to $H_{3,1}$, and $y^{q_0}$ is proportional to $H_{2,3}$. This is summarized in Table \ref{table:H}.
\begin{table}[htb!]
\[
\begin{array}{llll}
f =     &1   &\quad f^{q_0} \sim  &H_{1,2} =[1,x]\\
          &x &              &H_{1,3}=[1,y] \\
          &y &              &H_{2,3}=[x,y] \\
\end{array}
\]
\caption{The Pl\"ucker Coordinates of the Hermitian Curve.}\label{table:H}
\end{table}

Now we give a visual way to list the equation of the Hermitian curve using a complete graph with three vertices. The advantage of using a complete graph to read the equation of the curve is that it will be easy to generalize the same interpretation later for the Suzuki and Ree curves. We construct a triangle with vertices corresponding to the functions $1,x,y$ and the edge between any two vertices is labeled by the function that corresponds to the Pl\"{u}cker coordinate of the two vertices in Table \ref{table:H}, i.e., the function $f$ with $f^{q_0}\sim H_{i,j}$ in Table \ref{table:H}. For example, the edge between $1,x$ is labeled with $1$ because we have $1^{q_0}\sim H_{1,2}$ in Table \ref{table:H}. Therefore, we get the following graph

\begin{figure}[htb!]
\centering

\begin{tikzpicture}
  [scale=1,auto=left,]
  \node [circle,draw] (a) at (1,3) {$1$};
  \node [circle,draw] (b) at (3,4)  {$x$};
  \node [circle,draw] (c) at (3,2)  {$y$};

  \draw (a)[->] -- node {$1$} (b);
  \draw (a)[->] -- node[swap]{$x$} (c);
  \draw (b)[->] -- node {$y$} (c);

\end{tikzpicture}
\caption{The complete graph with three vertices.}\label{figure:32}
\end{figure}
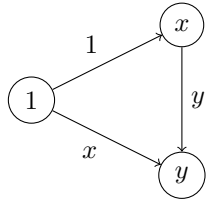

Now we can read the equation of the Hermitian curve as follows. We raise the vertices to the power of $q_0$ and we multiply them by the opposite edge and we sum the result to get
\[
 y\cdot 1^{q_0}-x\cdot x^{q_0}+1\cdot y^{q_0}=0,
\]
where the clockwise orientation is the positive orientation, i.e., the arrows in the clockwise direction are positive and the arrows in the counter clockwise direction are negative.

\begin{remark}
 The motivation to use the complete graph with three vertices is the two dimensional root system in Figure~\ref{table1:RootSystem} (a) by taking the short roots $\alpha$ to be $1$, $-\alpha$ to be $y$, and the origin to be $x$.
\end{remark}

 Next, we carry out this construction for the Suzuki curve which has a model with a similar description using Pl\"{u}cker coordinates. To see the correspondence between these techniques we denote the function $1$ by $x_0$, $x$ by $x_1$, and $y$ by $x_2$.

\subsection{The Suzuki Curve}\label{sec1:3.2}

The Suzuki curve has been studied in detail in \cite{TGK},\cite{HS},\cite{Henn},\cite{Tor3}. It is defined over the finite field $\fgq$ of characteristics 2, where $q:=2q_0^2=2^{2m+1}$ $(m \in \NN)$, and corresponds to the Suzuki function field $\fs:=\fgq(x,y)$ defined over $\fgq$ by the affine equation $y^q-y=x^{q_0}(x^q-x)$ \cite{HS}. The Suzuki function field $\fs \fieldextension \fgq$ has $q^2+1$ $\fgq$-rational places with one place at infinity $P_\infty$ and is of genus $\gs=q_0(q-1)$. Therefore, the Suzuki function field is optimal with respect to Serre's explicit formula method. 
The number of $\fg{q^r}$-rational places is given by
\begin{equation} \label{eq:Suzfqm}
 N_r=q^r+1-\gs q_0^r\left [ (-1+i)^r +(-1-i)^r \right ].
\end{equation}
In particular, the Suzuki curve is maximal if $r \equiv 0 \pmod{4}$. Moreover, the Suzuki function field is the unique function field of genus $q_0(q-1)$ and $q^2+1$ $\fgq$-rational places over $\fgq$ \cite[Theorem 5.1]{FTor1}. The automorphism group of $\fs \fieldextension \fgq$ is the Suzuki group $\BB=\Sz(q)$ \cite{Henn} of order $q^2(q-1)(q^2+1)$. For this reason it is known as the Suzuki curve.

From the above, the Suzuki curve $\xs$ has a projective irreducible plane model defined in $\PP^2(\fgqc)$ by the homogeneous equation
\[
 t^{q_0}(y^q-yt^{q-1})=x^{q_0}(x^q-xt^{q-1})
\]
which is a curve with a singularity only at the point at infinity $P_\infty=[0:0:1]$.

As in \cite{HS}, let $\nu_\infty$ be the discrete valuation of $\fs$ at the place $P_\infty$ and define two more functions $z:=x^{2q_0+1}-y^{2q_0}$ and $w:=xy^{2q_0}-z^{2q_0}$. Then, we have that the functions $x,y,z$, and $w$ are regular outside $P_\infty$ with pole orders at $P_\infty$ as in Table \ref{table:3.1}.

\begin{table}[htb!]
\begin{center}
\begin{tabular}{cccccc}
\toprule
$f$ & 1 & $x$ & $y$ & $z$ & $w$   \\
\midrule
$-\nu_\infty(f)$ & 0 & $q$ & $q+q_0$ & $q+2q_0$ & $q+2q_0+1$ \\
\bottomrule
\end{tabular}
\end{center}
\caption{The pole orders of $1,x,y,z,w$ at $P_\infty$.}
\label{table:3.1}
\end{table}

Moreover, the monoid $\langle q,\, q+q_0,\, q+2q_0,\, q+2q_0+1 \rangle$ is equal to the Weierstrass non-gaps semigroup $H(P_\infty)$ \cite{HS}.

To find a smooth model for the Suzuki curve $\xs$, let $\mathcal{S}$ be its normalization. Giulietti, K\'orchm\'aros, and Torres \cite{TGK} used the divisor $H:=mP_\infty$, where $m:=q+2q_0+1=-\nu_\infty(w)=h(1)$ (where $h(t)=q+2q_0t+t^2\in \ZZ[t]$ is the product of the irreducible factors of the reciprocal of the $L$-polynomial of $\xs$) and considered the complete linear series $\Da:=\left | (q+2q_0+1)P_\infty \right| $. Then, we have
\begin{prop}\label{prop:3.1}
With the notations above, we have:
\begin{enumerate}
 \item $\LL(mP_\infty)$ is generated by $1,x,y,z,w$ and so $\Da$ has dimension 4.
 \item $\Da$ is a very ample linear series.
 \end{enumerate}
\end{prop}

\begin{proof}
See \cite{HS} and Theorem 3.1 in \cite{TGK}.
\end{proof}
Using Proposition \ref{prop:3.1} above, we get a smooth embedding
\begin{align*}
 \pi : \mathcal{S} &\to \PP^4(\fgqc) \\
       P &\mapsto (1:x:y:z:w).   
\end{align*}

Now we give a concrete realization of the smooth embedding for the Suzuki curve in the projective space from \cite{TGK}. Since $y=x^{q_0+1}-z^{q_0}$ and $w=x^{2q_0+2}-xz+z^{2q_0}$ ($\ast$), define the embedding of $\mathcal{S}$ to be the variety in $\PP^4(\fgqc)$ defined by the set of points
\[
 P_{(a,c)}:=(1: a:b :c:d) \text{ and } \pi(P_\infty)=(0:0:0:0:1)
\]
where $x=a,z=c \in \fgqc$, and $y=b,w=d \in \fgqc$ are satisfying the two Equations ($\ast$) above. Moreover, the Suzuki group acts linearly on $\xs$ if it is considered as a subgroup of Aut($\PP^4(\fgqc)$) \cite[Theorem 3.2]{TGK}.

In this section we consider a different approach to construct a smooth model for the Suzuki curve which is similar to the Hermitian curve (Remark \ref{ch1:hermitianremrk}). The idea is to use the Pl\"ucker coordinates of the unique line between a point and its Frobenius image. The construction will be applied later for the Ree curve in Section \ref{sec1:4}. We describe it in terms of general variables. Let $x_{-2},x_{-1},x_{0},x_{1},x_{2}$ be the functions $t=1,x,y,z,w$, respectively. A preliminary results for the Suzuki curve announced in \cite{DEmail},\cite{DTalk} came up with the following system of equations
\begin{align*}
x_0^2+x_{-1}x_{1}+x_{-2}x_{2}=0\\
\begin{pmatrix}
 0 & x_{-2} & x_{-1} & x_{0} \\
 x_{-2} & 0 & x_{0} & x_{1} \\
 x_{-1} & x_{0} & 0 & x_{2} \\
 x_{0} & x_{1} & x_{2} &0
\end{pmatrix}
 \begin{pmatrix}
 x_{2}\\
 x_{1}\\
 x_{-1}\\
 x_{-2}
\end{pmatrix}^{(q_0)}
=0,
\end{align*}

i.e.,
\begin{align}
&y^2+xz+tw=0,\label{eq:s5}\\
&tz^{q_0}+xx^{q_0}+yt^{q_0}=0,\label{eq:s1}\\
&tw^{q_0}+yx^{q_0}+zt^{q_0}=0,\label{eq:s2}\\
&xw^{q_0}+yz^{q_0}+wt^{q_0}=0,\label{eq:s3}\\
&yw^{q_0}+zz^{q_0}+wx^{q_0}=0.\label{eq:s4}
\end{align}

\begin{lemma}\label{sec2:lemma5}
 The five equations above define the Suzuki curve.
\end{lemma}

\begin{proof}
 To see that these five equations define the Suzuki curve, we need to show the following equations:
\begin{align*}
 z&=x^{2q_0+1}-y^{2q_0},\\
 w&=xy^{2q_0}-z^{2q_0},\\
 y^q-y&=x^{q_0}(x^q-x).
\end{align*}
To get the first equation $z=x^{2q_0+1}-y^{2q_0}$, we multiply Equation \eqref{eq:s1} by $x^{q_0}$ and we add it to Equation \eqref{eq:s2} to get
\begin{equation}\label{eq:lemma*}
 w^{q_0}+z+x^{q_0}z^{q_0}+x^{2q_0+1}=0
\end{equation}
Now Equation \eqref{eq:s5} yields that $y^{2q_0}=x^{q_0}z^{q_0}+w^{q_0}$, substituting that in \eqref{eq:lemma*},
\[
 y^{2q_0}+z + x^{2q_0+1}=0 \Rightarrow z=x^{2q_0+1}-y^{2q_0}.
\]

To get the second equation $w=xy^{2q_0}-z^{2q_0}$, we multiply Equation \eqref{eq:s1} by $z^{q_0}$ and we add the result to Equation \eqref{eq:s3} to get
\begin{align*}
 z^{2q_0}+x^{q_0+1}z^{q_0}+xw^{q_0}+w &=0\\
 z^{2q_0}+x(x^{q_0}z^{q_0}+w^{q_0})+w &=0\\
 z^{2q_0}+xy^{2q_0} +w &=0 \Rightarrow w=xy^{2q_0}-z^{2q_0}.
\end{align*}

Finally, we show the last equation $y^q-y=x^{q_0}(x^q-x)$ as follows
\begin{align*}
 y^q-y = (y^{2q_0})^{q_0}-y &= (z+x^{2q_0+1})^{q_0} - (z^{q_0}+x^{q_0+1})\\
                            &= z^{q_0}+x^{q+q_0}-z^{q_0}-x^{q_0+1}\\
                            &= x^{q+q_0}-x^{q_0+1}\\
                            &= x^{q_0}(x^q-x). \qedhere
\end{align*}
\end{proof}

\begin{remark}
 In \cite{TGK}, the authors used only Equations \eqref{eq:s1}, \eqref{eq:s3} to define the Suzuki curve, see Equations ($\ast$). Lemma \ref{sec2:lemma5} shows that the five equations form a complete set of equations to define the Suzuki curve.
\end{remark}

\begin{remark}\label{ch1:suzukiremrk}
From the five Equations \eqref{eq:s5}--\eqref{eq:s4} above, it follows that
\[
 \begin{pmatrix}
   0 & t^{2q_0} & x^{2q_0} & y^{2q_0} \\
   t^{2q_0} & 0 & y^{2q_0} & z^{2q_0} \\
   x^{2q_0} & y^{2q_0} & 0 & w^{2q_0} \\
   y^{2q_0} & z^{2q_0} & w^{2q_0} & 0
 \end{pmatrix}
\begin{pmatrix}
 w & w^q \\
 z & z^q \\
 x & z^q \\
 t & t^q
\end{pmatrix} = 0.
\]
Consider the following matrix $S$
\[
S = \left( \begin{array}{lllllll}
1 &:~x &:~z &:~w \\ \medskip
1 &:~x^q &:~z^q &:~w^q
\end{array} \right).
\]
Let $S_{i,j}$ be the Pl\"ucker coordinates of the matrix $S$. Then, we also have
\begin{equation}\label{eq:MatS}
 \begin{pmatrix}
   0 & S_{1,2} & S_{1,3} & S_{3,2} \\
   S_{1,2} & 0 & S_{1,4} & S_{4,2} \\
   S_{1,3} & S_{1,4} & 0 & S_{4,3} \\
   S_{3,2} & S_{4,2} & S_{4,3} & 0 \\
 \end{pmatrix}
\begin{pmatrix}
 w & w^q \\
 z & z^q \\
 x & x^q \\
 t & t^q
\end{pmatrix} = 0.
\end{equation}
As the line between a point and its Frobenius image is unique. Then, as for the Hermitian curve, we obtain the correspondence in Table \ref{table:S}.
\begin{table}[htb!]
\[
\begin{array}{llll}
f =    &1   &\quad f^{2q_0} \sim  &S_{1,2}=[1,x] \\
       &x &              &S_{1,3}=[1,z] \\
       &y &              &S_{1,4} = S_{3,2} ([1,w]=[z,x]) \\
       &z &              &S_{4,2}=[w,x] \\
       &w &              &S_{4,3}=[w,z] \\
\end{array}
\]

\caption{The Pl\"ucker Coordinates of the Suzuki Curve.}\label{table:S}
\end{table}
\end{remark}

Now we give a visual way to list the defining equations of the Suzuki curve $\xs$ from a complete graph.

Consider the complete graph with four vertices labeled by $x_{-2},x_{-1},x_{1},x_{2}$ as in Figure \ref{figure1:37}, where the edge between any two vertices $x_i,x_j$ is labeled by the function that corresponds to the Pl\"{u}cker coordinate of $x_i,x_j$ in Table \ref{table:S}. For example, the edge between $x_{-2}=1$ and $x_2=w$ is labeled by the function that corresponds to the Pl\"{u}cker coordinate $S_{1,4} $ which is $x_0=y$.

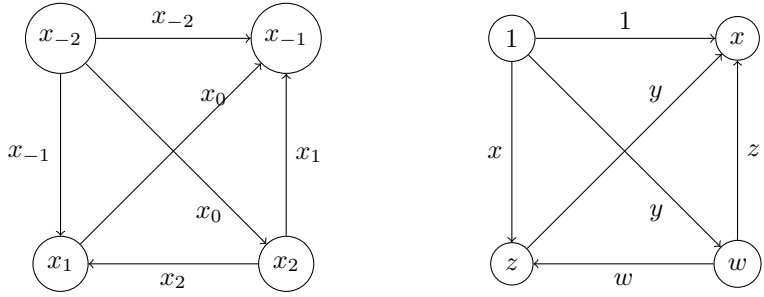
\begin{figure}[htb!]
\centering
\begin{tikzpicture}
  [scale=1,auto=left,]
  \node [circle,draw] (x1) at (1,1) {$x_{1}$};
  \node [circle,draw] (xm2) at (1,4)   {$x_{-2}$};
  \node [circle,draw] (x2) at (4,1)  {$x_{2}$};
  \node [circle,draw] (xm1) at (4,4)  {$x_{-1}$};

  \draw [->] (xm2) -- node[swap]{$x_{-1}$} (x1);
  \draw [->] (x2) -- node{$x_2$} (x1);
  \draw [->] (xm2) -- node[xshift=2mm,yshift=5.2mm]{$x_0$} (x2);
  \draw [->] (x1) -- node[swap,xshift=2.1mm,yshift=-5.8mm]{$x_{0}$} (xm1);
  \draw [->] (xm2) -- node{$x_{-2}$} (xm1);
  \draw [->] (x2) -- node[swap]{$x_1$} (xm1);

  [scale=1,auto=left,]
  \node [circle,draw] (x1) at (7,1) {$z$};
  \node [circle,draw] (xm2) at (7,4)   {$1$};
  \node [circle,draw] (x2) at (10,1)  {$w$};
  \node [circle,draw] (xm1) at (10,4)  {$x$};

  \draw [->] (xm2) -- node[swap]{$x$} (x1);
  \draw [->] (x2) -- node{$w$} (x1);
  \draw [->] (xm2) -- node[xshift=2mm,yshift=5.2mm]{$y$} (x2);
  \draw [->] (x1) -- node[swap,xshift=2.1mm,yshift=-5.8mm]{$y$} (xm1);
  \draw [->] (xm2) -- node{$1$} (xm1);
  \draw [->] (x2) -- node[swap]{$z$} (xm1);
\end{tikzpicture}
\caption{The complete graph with four vertices.}\label{figure1:37}
\end{figure}

Now to get the equations of total degree $q_0+1$ \eqref{eq:s1}--\eqref{eq:s2}, we consider any triangle in the polygon. We raise every vertex in the triangle to the power $q_0$ and we multiply it by the label of the opposite edge. Then we add them all to the equations of total degree $q_0+1$, e.g., if we consider the triangle in Figure \ref{figure2:37}, then we get the equation $aA^{q_0}+bB^{q_0}+cC^{q_0}=0$.

\begin{figure}[htb!]
\centering
\begin{tikzpicture}
  [scale=1,auto=left,]
  \node [circle,draw] (a) at (1,3) {$A$};
  \node [circle,draw] (b) at (3,4)  {$B$};
  \node [circle,draw] (c) at (3,2)  {$C$};

  \draw [->](a) -- node {$c$} (b);
  \draw [<-](a) -- node[swap]{$b$} (c);
  \draw [->](b) -- node {$a$} (c);

\end{tikzpicture}
\caption{Triangle.}\label{figure2:37}
\end{figure}
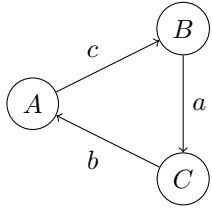
Therefore, we get the following four equations
\begin{alignat*}{3}
&x_{0}x^{q_0}_{-2}+x_{-1}x^{q_0}_{-1}+x_{-2}x^{q_0}_{1}=0,     \quad  &&tz^{q_0}+xx^{q_0}+yt^{q_0}=0 \qquad \qquad \quad &\eqref{eq:s1},\\ 
&x_{2}x^{q_0}_{-2}+x_{0}x^{q_0}_{1}+x_{-1}x^{q_0}_{2}=0,      \qquad\qquad \text{                         i.e.,} \qquad \qquad  		 &&tw^{q_0}+yx^{q_0}+zt^{q_0}=0 \quad\qquad \qquad &\eqref{eq:s2},\\ 
&x_{1}x^{q_0}_{-2}+x_{0}x^{q_0}_{-1}+x_{-2}x^{q_0}_{2}=0 ,     \quad  &&xw^{q_0}+yz^{q_0}+wt^{q_0}=0 \qquad \qquad \quad &\eqref{eq:s3},\\
&x_{2}x^{q_0}_{-1}+x_{1}x^{q_0}_{1}+x_{0}x^{q_0}_{2}=0,        \quad  &&yw^{q_0}+zz^{q_0}+wx^{q_0}=0 \qquad \qquad \quad &\eqref{eq:s4}.
\end{alignat*}

Since we have only four triangles in the polygon of four vertices, the equations above are the only equations of degree $q_0+1$.

To get the quadratic equation, we consider the polygon of four vertices. The product of the two diagonals plus the sum of the product of the opposite edges is equal to $0$, i.e.,
\[
 x_0^2+x_{-1}x_1+x_{-2}x_2=0, \qquad \text{i.e.,}  \qquad y^2+xz+tw=0.
\]

Finally, we note that we can get the quadratic equation (up to a $q_0$ power of a hyperplane) from the four equations of degree $q_0+1$ as follows. We fix any vertex, say $x_{-2}$ and we multiply each equation that contains $x_{-2}^{q_0}$ by the label of the missing edge of $x_{-2}$ in the triangle that define that equation, i.e., we multiply the \eqref{eq:s1} by $x_{0}$ (because $x_0$ is the missing edge of $x_{-2}$ in the triangle $\Delta x_{-2}x_{-1}x_{0}$ inside the polygon), \eqref{eq:s2} by $x_{-2}$, and \eqref{eq:s3} by $x_{-1}$ to get
\begin{align*}
 &x^2_{0}x^{q_0}_{-2}+x_{0}x_{-1}x^{q_0}_{-1}+x_{0}x_{-2}x^{q_0}_{1}=0,\\  
&x_{-2}x_{2}x^{q_0}_{-2}+x_{-2}x_{0}x^{q_0}_{1}+x_{-2}x_{-1}x^{q_0}_{2}=0.\\    
&x_{-1}x_{1}x^{q_0}_{-2}+x_{-1}x_{0}x^{q_0}_{-1}+x_{-1}x_{-2}x^{q_0}_{2}=0.
\end{align*}
Now we add the equations above, we get that $(x_0^2+x_{-1}x_1+x_{-2}x_2)x^{q_0}_2=0$. In other words, the four equations of degree $q_0+1$ define a reducible variety with one component the Suzuki curve and the remaining components are intersections of hyperplanes.

\begin{remark}
Another way to get the quadratic equation is to notice that the matrix in \eqref{eq:MatS} is singular. The determinant is a power of the the quadratic equation.
\end{remark}

\begin{remark}
 The motivation to use the graph above with four vertices is the two dimensional root system in Figure~\ref{table1:RootSystem} (b) by taking only the short roots $\alpha$, $\alpha+\beta$, $-\alpha$, and $-(\alpha+\beta)$.
\end{remark}

In Section \ref{sec1:4} we will apply the techniques that are used in this section to get a set of equations that define a smooth model for the Ree curve in $\PP^{13}(\fgqc)$. We mention that, although there are similarities between these Deligne-Lusztig curves in constructing a very ample linear series and smooth embeddings, there are some differences, for example, in the Weierstrass non-gaps semigroup at $P_\infty$. More specifically, the pole orders of the functions that give the smooth embeddings for the Hermitian or Suzuki curve generate the Weierstrass non-gaps semigroup at $P_\infty$ while this is not the case for the Ree curve (see Section \ref{sec1:8.1}).




\section{The Defining Equations and Automorphism group of the Ree Curve}\label{sec1:4}

\subsection{The Defining Equations for the Ree Curve}

In this section we give 105 equations that define a smooth model for the Ree curve in the projective space $\ppfgqc$. The embedding will be similar to the embeddings of the Hermitian and Suzuki curves using 14 functions that were defined by Pedersen \cite{Ped}. The general idea of this section is to apply the observations made in Sections \ref{sec1:3.1} and \ref{sec1:3.2} about the Hermitian and Suzuki curves.

We start with some notations and results about the Ree curve from \cite{Ped}. Let $m \in \NN$, $q_0:=3^m$, and $q:=3q_0^2$. Define the Ree function field $\fr:=\fgq(x,y_1,y_2)$ over $\fgq$ by the equations:
\begin{align}
 y_1^q-y_1 &=x^{q_0}(x^q-x),\label{eq:Ree1}\\
 y_2^q-y_2 &=x^{q_0}(y_1^q-y_1).\label{eq:Ree2}
\end{align}

The Ree function field $\fr \fieldextension \fgq(x)$ is a finite separable field extension of degree $q^2$. All affine rational places $Q_{a}=(x=a)\in \ppfx (a \in \fgq)$ split completely into $q^2$ rational places $P_{a,b,c}=(x=a,y_1=b,y_2=c)\in \ppfr$. Let $Q_{\infty}\in \ppfx$ be the pole of $x$ in $\fgq(x)$ and $P_{\infty}\in \ppfr$ be a place of $\fr$ lying above $Q_{\infty}$. Then, $P_{\infty}\placeextension Q_{\infty}$ is totally ramified in $\fr \fieldextension \fgq(x)$ with ramification index $e_{\infty}:=e(P_{\infty}\placeextension Q_{\infty})=q^2$. Thus, the function field $\fr$ has $q^3+1$ $\fgq$-rational places and is of genus $\gr=(\sfrac{3}{2})q_0(q-1)(q+q_0+1)$. Therefore, it is an optimal curve with respect to Serre's explicit formula method. The number of $\fg{q^r}$-rational places is given by 
\begin{equation} \label{eq:Reefqm}
 N_r=q^r+1-\sqrt{3}q_0^2(q-1)\left [ (q+q_0+1)\cos \left ( \frac{r\pi}{6}\right ) +2(q+1)\cos \left ( \frac{5r\pi}{6} \right )    \right ].
\end{equation}
In particular, the Ree curve is maximal if $r \equiv 6 \pmod{12}$.

The automorphism group of this function field is the Ree group $R(q)=\,\GG(q)$ of order $q^3(q-1)(q^3+1)$. For this reason it is known as the Ree function field. Moreover, by the result of Hansen and Pedersen \cite{HP}, the Ree function field is the unique function field of genus $\gr=\sfrac{3q_0(q-1)(q+q_0+1)}{2}$, number of $\fgq$-rational places equals to $q^3+1$, and automorphism group $\GG$. Denote by $\nu_0$ and $\nu_{\infty}$ the valuations at the places $P_{000}$ and $P_{\infty}$ respectively. Pedersen \cite{Ped} constructed the following ten functions $w_1,\dots,w_{10}\in \fr$:

\begin{minipage}{.45\linewidth}
\begin{align}
& w_1:=x^{3q_0+1} - y_1^{3q_0},\label{ReeEq:1}\\
& w_2:=xy_1^{3q_0}-y_2^{3q_0},\label{ReeEq:2}\\
& w_3:=xy_2^{3q_0}-w_1^{3q_0},\label{ReeEq:3}\\
& w_4:= xw_2^{q_0}-y_1w_1^{q_0},\label{ReeEq:4}\\
& v:=xw_3^{q_0}-y_2w_1^{q_0},\label{ReeEq:v}\\
 & w_5:=y_1w3_{q_0}-y_2w_1^{q_0},\label{ReeEq:5}\\ \notag
\end{align}
\end{minipage}%
\begin{minipage}{.47\linewidth}
\begin{alignat}{2}
 & w_6:&&=v^{3q_0}-w_2^{3q_0}+xw_4^{3q_0},\label{ReeEq:6}\\
 & w_7:&&=y_1w_3^{q_0}-xw_3^{q_0}-w_6^{q_0}\label{ReeEq:7} \\
 & \, &&=w_2+v, \notag \\
 & w_8:&&= w_6^{3q_0}+xw_7^{3q_0},\label{ReeEq:8}\\
 & w_9:&&=w_4w_2^{q_0} -y_1w_6^{q_0},\label{ReeEq:9}\\
 & w_{10}:&&=y_2w_6^{q_0}-w_3^{q_0}w_4.\label{ReeEq:10} \\ \notag
\end{alignat}

\end{minipage}

We remark here that some of these equations were already used by Tits \cite{Tits} to describe the generators of the Ree group and to show it is a simple group acting 2-transitively on a set of $q^3+1$ points.

\begin{remark}\label{sec1:ZeroPole}
We classify the 14 functions $1,x,y_1,y_2,w_1,\dots,w_{10}$ into two sets of variables denoted by $x_{-3}, x_{-2},\dots,x_{3}$, and $y_{-3}, y_{-2}, \dots, y_{3}$, as illustrated in Table \ref{table:4.2}. We also have the auxiliary functions $z_1:=w_7$, $z_2:=w_7+w_2$, and $z_3:=w_7-w_2$ which satisfy $z_1+z_2+z_3=0$.

\begin{table}[htb!]
\begin{center}
\begin{tabular}{  c  c  c c  c c c}
\toprule
$x_{-3}$ & $x_{-2}$ & $x_{-1}$ & $x_{0}$ & $x_{1}$ & $x_{2}$ & $x_{3}$ \\
\midrule
$w_{1}$ & $x$ & $-w_8$ & $w_2$ & $1$ & $-w_6$ & $-w_3$ \\
\midrule

\midrule
$y_{-3}$ & $y_{-2}$ & $y_{-1}$ & $y_{0}$ & $y_{1}$ & $y_{2}$ & $y_{3}$ \\
\midrule
$w_{4}$ & $-y_{2}$ & $-w_{10}$ & $w_7$ & $y_{1}$ & $w_{9}$ & $-w_{5}$ \\
\bottomrule
 \end{tabular}
\end{center}
\caption{The correspondence between the 14 functions and $x_i$'s, $y_i$'s.}
\label{table:4.2}
\end{table}
Consider the involution automorphism $\phi:\fr \to \fr$ acting by $x_{i}\mapsto \sfrac{x_{-i}}{x_{-1}}$ and $y_{i}\mapsto \sfrac{y_{-i}}{x_{-1}}$ ($i=0,1,2,3$). The automorphism $\phi$ sends the place $P_{000}$ to the place $P_{\infty}$. Therefore, the pole order and the zero order of $x_i$ (resp. $y_i$) and $x_{-i}$ (resp. $y_{-i}$) are related by
\[
 \nu_{0}(x_i)=-\nu_{\infty}(w_8)+\nu_{\infty}(x_{-i})\quad \text{resp.}\quad  \nu_{0}(y_i)=-\nu_{\infty}(w_8)+\nu_{\infty}(y_{-i}).
\]
The valuations of the 14 functions at $P_{000}$ and $P_{\infty}$ are summarized in Table \ref{table:4.1}.

\end{remark}

The Ree curve is birationally equivalent to the projective curve in $\pp^{3}(\fgqc)$ defined by the $2\times 2$-minors of the matrix
\[
 \begin{pmatrix}
X^{q_0} & -(Y_1^{q}-Y_{1}U^{q-1}) & -(Y_2^{q}-Y_{2}U^{q-1})\\
U^{q_0} & -(X^{q}-XU^{q-1}) & -(Y_1^{q}-Y_{1}U^{q-1})
\end{pmatrix}.
\]
This curve has a singularity at the point at infinity $[0:0:0:1]$ which corresponds to the place $P_{\infty}$ of $\fr$. Moreover, the Ree curve has a singular plane model in which $\fr$ is defined as an Artin-Schreier extension in the variables $x$ and $w_2$ (see \cite{Ped}).

One central problem from Pedersen paper \cite{Ped} is the following problem.

\begin{problem}
Compute the Weierstrass non-gaps semigroup $H(P_\infty)$.
\end{problem}

In order to solve the problem above, we have found 105 equations in $1,x,y_1,y_2$, $w_1,\dots,w_{10}$. These equations are then used to compute all the non-gaps at $P_\infty$ over $\fg{27}$. (see Section \ref{sec1:8}). Moreover, these equations define a smooth embedding for the Ree curve in $\ppfgqc$ using the functions $1,x,y_1,y_2,w_1,\dots,w_{10}$. These equations which are listed in Appendix \ref{app:1} can be described as follows:


\begin{description}
 \item[Set 1] Equations of total degree $q_0+1$ of the form $aA^{q_0}+bB^{q_0}+cC^{q_0}=0$, where the functions $A,B,C \in \{ 1,x,w_1,w_2,w_3,w_6,w_8 \}$ and $a,b,c \in \{1,x,y_1,y_2,w_1,\dots,w_{10} \}$.

  \item[Set 2] Equations of total degree $3q_0+1$ of the form $a^{3q_0}A+b^{3q_0}B+c^{3q_0}C=0$, where $a,b,c,A,B,C$ are the functions in Set 1.

  \item[Set 3] One quadratic equation $-w_2^2+w_8+xw_6+w_1w_3=0$.

  \item[Set 4] Quadratic equations.
\end{description}

\begin{lemma} \label{sec3:lemma105}
The 105 equations in Appendix \ref{app:1} define the Ree curve.
\end{lemma}
\begin{proof}
To show that the 105 equations define the Ree curve, we need a birational map between the two models. From the 105 equations we show that the following equations hold:
\begin{align*}
 y_1^{q}-y_1&=x^{q_0}(x^q-x),\\
 y_2^q-y_2&=x^{q_0}(y_1^q-y_1).
\end{align*}
To get the first equation $y_1^{q}-y_1=x^{q_0}(x^q-x)$, we use the two equations $y_1=x^{q_0+1}-w_1^{q_0}$ \eqref{app:eq8} and $w_1=x^{3q_0+1}-y_1^{3q_0}$ \eqref{app:eq8'}. Then, we have
\begin{align*}
 y_1^q-y_1-x^{q+q_0}+x^{q_0+1}&=y_1^{3q_0^2}-y_1-x^{3q_0^2+q_0}+x^{q_0+1}\\
                              &=(y_1^{3q_0}-x^{3q_0+1})^{q_0}+(x^{q_0+1}-y_1)\\
                              &=(-w_1)^{q_0}+w_1^{q_0}\\
                              &=0.
\end{align*}

Similarly, to get the second equation $y_2^q-y_2=x^{q_0}(y_1^q-y_1)$, we use the two equations $y_2=y_1x^{q_0}-w_2^{q_0}$ \eqref{app:eq4} and $w_2=xy_1^{3q_0}-y_2^{3q_0}$ \eqref{app:eq4'}. Then, we have
\begin{align*}
 y_2^q-y_2-x^{q_0}y_1^q+x^{q_0}y_1&=y_2^{3q_0^2}-y_2-x^{q_0}y_1^{3q_0^2}+y_1x^{q_0}\\
                                  &=(y_2^{3q_0}-xy_1^{3q_0})^{q_0}+(y_1x^{q_0}-y_2) \\
                                  &=(-w_2)^{q_0}+w_2^{q_0}\\
                                  &=0.
\end{align*}
Finally, it is easy to see that the Ree curve satisfy the 105 equations.
\end{proof}

\begin{remark}
Consider the following matrix $R$
\[
R = \left( \begin{array}{lllllll}
1 &:~x &:~w_1 &:~w_2 &:~w_3 &:~w_6 &:~w_8  \\ \medskip
1 &:~x^q &:~w_1^q &:~w_2^q &:~w_3^q &:~w_6^q &:~w_8^q
\end{array} \right).
\]
Then, following the same ideas of Remark \ref{ch1:hermitianremrk} and Remark \ref{ch1:suzukiremrk} together with the equations in Set 2 in Appendix \ref{app:1}, we let a function $f^{3q_0}$ correspond to the Pl\"ucker coordinates of the matrix $R$ as in Table~\ref{table:R}.


\begin{table}[htb!]
\[
\begin{array}{llll}
f =    &1   &\quad f^{3q_0} \sim  &R_{1,2} = [1,x] \\
       &x &                &R_{1,3} = [1,w_1] \\
       &w_1 &              &R_{2,5} = [x,w_3] \\
       &w_3 &              &R_{6,3} = [w_6,w_1]\\
       &w_6 &              &R_{7,5} =[w_8,w_3]\\
       &w_8 &              &R_{7,6} =[w_8,w_6]\\
       &    &              &        \\
       &y_1 &              &R_{2,3} = R_{1,4}\, ([x,w_1]=[1,w_2]) \\
       &y_2 &              &R_{1,5} = R_{2,4}\, ([1,w_3]=[x,w_2])\\
       &w_4 &              &R_{1,6} = R_{4,3}\, ([1,w_6]=[w_2,w_1])\\
       &w_5 &              &R_{7,2} = R_{5,4}\, ([w_8,x]=[w_3,w_2])\\
       &w_9  &             &R_{7,3} = R_{4,6}\, ([w_8,w_1]=[w_2,w_6])\\
       &w_{10} &            &R_{6,5} = R_{4,7}\, ([w_6,w_3]=[w_2,w_8])\\
       &    &              &        \\
       &v_1 &              &R_{5,3}=[w_3,w_1]   \\
       &v_1+w_2 &          &R_{1,7}=[1,w_8]   \\
       &v_1-w_2 &          &R_{6,2}=[w_6,x]   \\[.5ex]
\end{array}
\]
\caption{The Pl\"ucker coordinates of the Ree curve.}\label{table:R}
\end{table}

\end{remark}

\begin{remark}
Using Table \ref{table:R}, we can write the quadratic equations (Set 4) in the form $f_{R_{ab}}f_{R_{cd}}+f_{R_{ad}}f_{R_{bc}}+f_{R_{ac}}f_{R_{db}}=0$, where $f_{R_{ab}}$ is the function such that $f_{R_{ab}}^{3q_0} \sim R_{ab}$ in Table \ref{table:R}.
\end{remark}

Now we give a visual (geometric) way to list the equations above, which are in Sets 1--4 of the Ree curve. Consider the complete graph with seven vertices labeled by $x_{-3},x_{-2},\dots,x_{3}$, where the edge between any two vertices $x_i$ and $x_j$ is labeled by the function that corresponds to the Pl\"ucker coordinate of $x_i$ and $x_j$ in Table \ref{table:R}. For example, the edge between $x_0=w_2$ and $x_3=-w_3$ is labeled by the function that corresponds to the Pl\"ucker coordinate $R_{4,5}$ which is $-y_{3}=w_5$. \\

To list all the equations of the Ree curve, we will use the graph in Figure~\ref{fig:G7}. Note that the edges from $x_1,x_2,x_3$ to $x_{0},x_{-1},x_{-2},x_{-3}$ are outgoing edges, the edges from $x_0$ to $x_{-1},x_{-2},x_{-3}$ are outgoing edges, the edges between $x_1,x_2,x_3$ are according to the permutation $(1,3,2)$, and the edges between $x_{-1},x_{-2},x_{-3}$ are according to the permutation $(1,2,3)$. As a convention, the clockwise orientation will be considered as the positive orientation of this graph. \\

The labeling of the edges in Figure \ref{fig:G7} is given by Table \ref{table:4.3}. The labeling matches the Pl\"ucker coordinates for the matrix $R$ given in Table \ref{table:R}. \\

\begin{table}[htb!]
\begin{center}
\begin{tabular}{c c ccc ccc ccc }
\toprule
\renewcommand{\arraystretch}{2}
 &$~~~~$ & $1$ & $-w_{6}$ & $-w_{3}$ &$~~$ & $w_{2}$  &$~~$ & $-w_{8}$ & $x$ & $w_{1}$ \\[.5ex]
\midrule
$1$            & &           & $-w_{4}$ & $-y_2$ & & $y_{1}$ & & $-w_7$ & $1$ & $x$ \\[.5ex] 
$-w_{6}$   & & $w_{4}$ &  & $w_{10}$ & & $w_{9}$ & & $-w_{8}$ & $-w_7-w_2$ & $-w_{3}$  \\[.5ex] 
$-w_{3}$   & & $-y_{2}$ & $-w_{10}$ &  & & $-w_{5}$ & & $-w_6$ & $w_{1}$ & $-w_{7}+w_2$ \\[3mm]
$w_{2}$    & & $-y_{1}$  & $-w_{9}$ & $w_{5}$ & &  & & $-w_{10}$ & $-y_{2}$ & $w_{4}$ \\[3mm]
$-w_{8}$   & & $w_{7}$  & $w_{8}$ & $w_{6}$ & & $w_{10}$ & &  & $-w_{5}$ & $-w_9$ \\[.5ex]
$x$            & & $-1$  & $w_7+w_2$ & $-w_1$ & & $y_{2}$ & & $w_{5}$ &  & $y_{1}$ \\[.5ex]
$w_{1}$    & & $-x$        & $w_{3}$ & $w_7-w_2$ & & $-w_{4}$ & & $w_{9}$ & $-y_{1}$ & \\[.5ex]
\bottomrule \\
\end{tabular}
\end{center}
\caption{The edge labeling of the graph in Figure \ref{fig:G7}. }
\label{table:4.3}
\end{table}
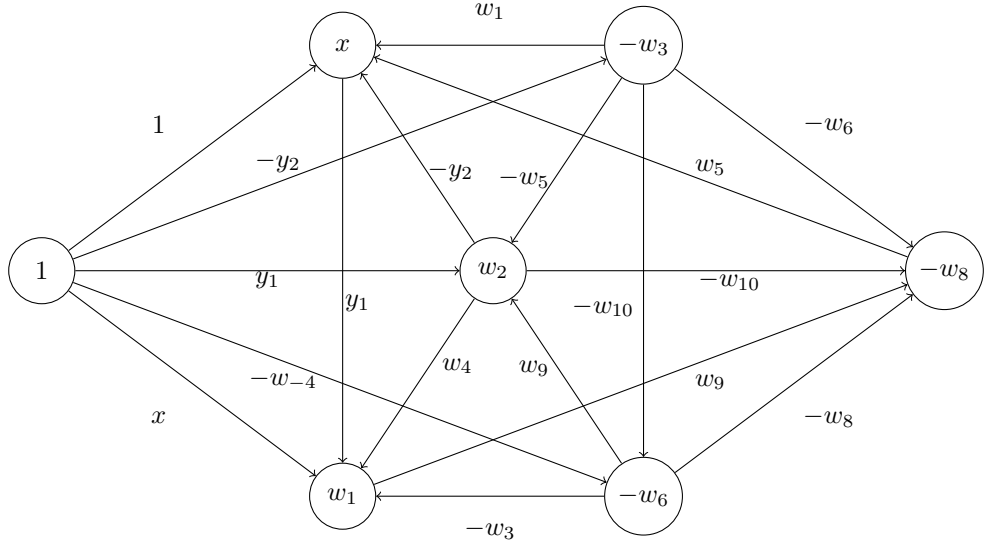
\begin{figure}[htb!]
\centering
\begin{tikzpicture}
  [scale=1,auto=left, every node/.style={minimum size=2.5em}]
  \node [circle,draw] (xp1) at (1,4) {$1$};
  \node [circle,draw] (xm3) at (5,1)  {$w_{1}$};
  \node [circle,draw] (xm2) at (5,7)  {$x$};
  \node [circle,draw] (x0) at (7,4) {$w_{2}$};
  \node [circle,draw] (xp2) at (9,1)  {$-w_{6}$};
  \node [circle,draw] (xp3) at (9,7)  {$-w_{3}$};
  \node [circle,draw] (xm1) at (13,4)  {$-w_{8}$};

  \draw [->] (xp1) -- node {$1$} (xm2);
  \draw [<-] (xm2) -- node {$w_1$} (xp3);
  \draw [->] (xp3) -- node {$-w_{6}$} (xm1);
  \draw [<-] (xm1) -- node {$-w_{8}$} (xp2);
  \draw [->] (xp2) -- node {$-w_{3}$} (xm3);
  \draw [<-] (xm3) -- node {$x$} (xp1);

  \draw [->] (x0) -- node[xshift=8.8mm,yshift=2.8mm] {$-y_{2}$} (xm2);
  \draw [->] (x0) -- node[xshift=.8mm,yshift=6.8mm] {$w_{4}$} (xm3);
  \draw [->] (x0) -- node[xshift=1.8mm,yshift=-5.8mm] {$-w_{10}$} (xm1);

  \draw [<-] (x0) -- node[xshift=-8.8mm,yshift=-2.8mm] {$w_{9}$} (xp2);
  \draw [<-] (x0) -- node[xshift=-1mm,yshift=-6.8mm] {$-w_{5}$} (xp3);
  \draw [<-] (x0) -- node[xshift=0mm,yshift=3.2mm] {$y_{1}$} (xp1);

  \draw [->] (xp1) -- node[xshift=-3.8mm,yshift=-4.8mm] {$-y_{2}$} (xp3);
  \draw [<-] (xp2) -- node[xshift=-1.8mm,yshift=4.8mm] {$-w_{-4}$} (xp1);
  \draw [<-] (xp2) -- node[xshift=0mm,yshift=-4.8mm] {$-w_{10}$} (xp3);

  \draw [<-] (xm1) -- node[xshift=4.8mm,yshift=4.8mm] {$w_{9}$} (xm3);
  \draw [<-] (xm2) -- node[xshift=4.8mm,yshift=-5.8mm] {$w_{5}$} (xm1);
  \draw [->] (xm2) -- node[xshift=-2.5mm,yshift=-4.3mm] {$y_{1}$} (xm3);

\end{tikzpicture}
\caption{The complete graph with the seven vertices.}\label{fig:G7}
\end{figure}
\quad\\
\begin{figure}[htb!]
\centering

\begin{tikzpicture}
  [scale=1,auto=left,minimum size=2.8em]

  \node [circle,draw] (xm3) at (1,1)  {$w_1$};
  \node [circle,draw] (xp3) at (3,7)  {$-w_{3}$};

  \node [circle,draw] (xp1) at (4,3.5) {$1$};
  \node [circle,draw] (xm1) at (8,3.5)  {$-w_{8}$};

  \node [circle,draw] (xm2) at (9,7)  {$x$};
  \node [circle,draw] (xp2) at (11,1)  {$-w_{6}$};


  \draw [->] (xp1) -- node {$-w_7$}  (xm1);
  \draw [<-] (xm3) -- node {$-w_7+w_2$} (xp3);
  \draw [->] (xp2) -- node[swap] {$-w_7-w_2$} (xm2);

\end{tikzpicture}
\caption{The diagonals in the complete graph with seven vertices in Figure~\ref{fig:G7}.}\label{fig2:G7}
\end{figure}
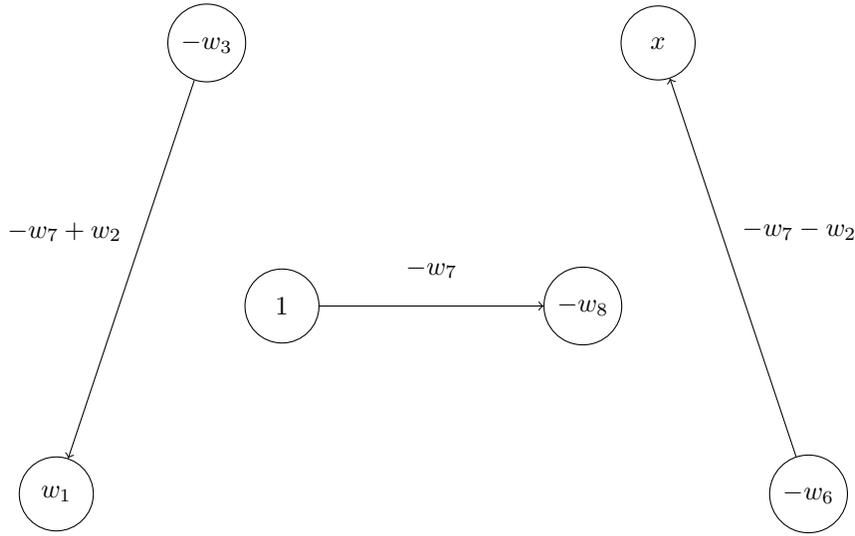

Now we use the graph in Figure \ref{fig:G7} to list the equations in Sets 1--4 as follows. To get the quadratic equations in Set 4, we consider the polygon in Figure \ref{figure:44}. Then, the product of the two diagonals plus the sum of the product of the opposite edges is equal to $0$, i.e.,
\[ad+be+cf=0.\] 

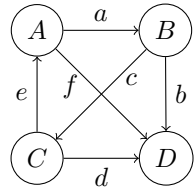
\begin{figure}[htb!]
\centering
\begin{tikzpicture}
  [scale=1,auto=left,]
  \node [circle,draw] (C) at (1,1) {$C$};
  \node [circle,draw, above=1cm of C] (A)   {$A$};
  \node [circle,draw, right=1cm of C] (D)   {$D$};
  \node [circle,draw, right=1cm of A] (B)   {$B$};

  \draw [->] (A) -- node{$a$} (B);
  \draw [<-] (A) -- node[swap]{$e$} (C);
  \draw [->] (A) -- node[xshift=2mm]{$c$} (D);
  \draw [->] (B) -- node[swap,xshift=-1.8mm,yshift=-1.8mm]{$f$} (C);
  \draw [->] (B) -- node{$b$} (D);
  \draw [->] (C) -- node[swap]{$d$} (D);
\end{tikzpicture}
\caption{Polygon with four vertices.}\label{figure:44}
\end{figure}

Now since we have seven vertices in the graph, in total we have $\binom{7}{4}=35$ polygons of four vertices, i.e., we have 35 equations. Among them, we notice that the equation
\[
 y_1w_{10}+y_2w_9+w_4w_5=0
\]
can be found from the two polygons in Figure \ref{figure1:45}:
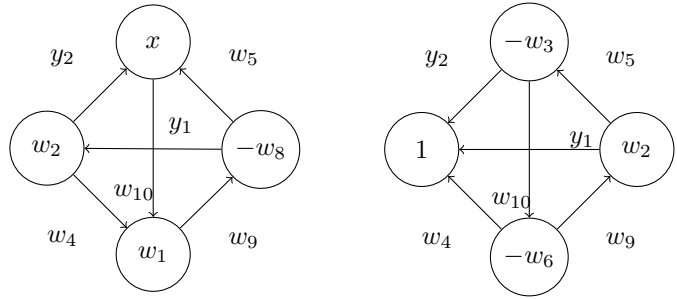
\begin{figure}[htb!]
\centering
\begin{tikzpicture}
  [scale=1,auto=left,minimum size=2.8em]
  \node [circle,draw] (x) at (5,5) {$x$};
  \node [circle,draw, below left=1cm of x] (w2)   {$w_2$};
  \node [circle,draw, below right=1cm of x] (w8)   {$-w_8$};
  \node [circle,draw, below right=1cm of w2] (w1)   {$w_1$};

  \draw [<-](x) -- node[swap]{$y_2$} (w2);
  \draw [<-](x) -- node{$w_5$} (w8);
  \draw [->](x) -- node[xshift=-1.3mm,yshift=2.9mm]{$y_1$} (w1);
  \draw [<-](w2) -- node[swap,xshift=-2.5mm,yshift=-0.8mm]{$w_{10}$} (w8);
  \draw [->](w2) -- node[swap]{$w_4$} (w1);
  \draw [<-](w8) -- node{$w_9$} (w1);

  [scale=1,auto=left,]
  \node [circle,draw] (x1) at (10,5) {$-w_3$};
  \node [circle,draw, below left=1cm of x1] (w22)   {$1$};
  \node [circle,draw, below right=1cm of x1] (w88)   {$w_2$};
  \node [circle,draw, below right=1cm of w22] (w11)   {$-w_6$};

  \draw [->](x1) -- node[swap]{$y_2$} (w22);
  \draw [<-](x1) -- node{$w_5$} (w88);
  \draw [->](x1) -- node[xshift=2mm,yshift=1.3mm]{$y_1$} (w11);
  \draw [<-](w22) -- node[swap,xshift=-2.3mm,yshift=-1.8mm]{$w_{10}$} (w88);
  \draw [<-](w22) -- node[swap]{$w_4$} (w11);
  \draw [<-](w88) -- node{$w_9$} (w11);
\end{tikzpicture}
\caption{The polygons yield the same quadratic equation.}\label{figure1:45}
\end{figure}

Hence, it appears twice in the list and so we have only 34 equations in Set 4. 

Next, to get the equations in Set 1 of degree $q_0+1$, we consider any triangle in the graph, we take every vertex in the triangle to the power $q_0$ and we multiply it with the opposite edge. Then, we add them all to get the equation of total degree $q_0+1$.

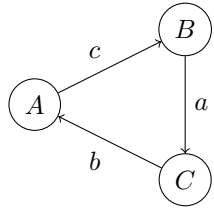
\begin{figure}[htb!]
\centering
\begin{tikzpicture}
  [scale=1,auto=left,]
  \node [circle,draw] (a) at (1,3) {$A$};
  \node [circle,draw] (b) at (3,4)  {$B$};
  \node [circle,draw] (c) at (3,2)  {$C$};

  \draw[->] (a) -- node {$c$} (b);
  \draw[<-] (a) -- node[swap]{$b$} (c);
  \draw[->] (b) -- node {$a$} (c);

\end{tikzpicture}
\caption{Triangle.}\label{figure2:45}
\end{figure}

For example, for the triangle on in Figure \ref{figure2:45}, we get $aA^{q_0}+bB^{q_0}+cC^{q_0}=0$. Note that all the arrows are in the positive orientation. If an arrow is in the negative orientation, then we multiply the edge by a negative sign.

Since we have $\binom{7}{3}=35$ triangles in the graph of seven vertices, in total we have 35 equations in Sets 1 and these will also give another 35 equations in Set 2 by $a^{3q_0}A+b^{3q_0}B+c^{3q_0}C=0$.

Moreover, we have the equation $1\cdot w_8 + x \cdot w_6 + w_1 \cdot w_3 - w^2_2=0$ which can be read from the long diagonals of the graph using the vertices $x_{-3},x_{-2},x_{-1},x_{1},x_{2},x_{3}$. Thus, in total we have $35+35+34+1=105$ equations.

\begin{remark}
 The motivation to use the graph above with seven vertices is the two dimensional root system in Figure~\ref{table1:RootSystem} (c) by taking the six short roots $\alpha$, $\alpha+\beta$, $2\alpha+\beta$, $-\alpha$, $-(\alpha+\beta)$, $-(2\alpha+\beta)$, and the origin to be the $x_i$'s in our notation and we take the long roots and the origin to be the $y_i$'s in our notation. These short roots and the origin give a seven dimensional representation of the Ree group.
\end{remark}

We mention that the ideal generated by these 105 equations can be generated (up to some power of a hyperplane) only by the first 35 equations from Set 1 as in the following lemma.

\begin{lemma}
Let $I \subseteq \fgq[x,y_1,y_2,w_1,\dots,w_{10}]$ be the ideal generated by the first 35 equations from Set 1 of degree $q_0+1$. Then, the equations in Set 2, 3, and 4 can be deduced (up to some power of a hyperplane) from $I$. In other words, the equations in Set 1 of degree $q_0+1$ define a reducible variety with one component the Ree curve and the remaining components are intersection of hyperplanes.
\end{lemma}

\begin{proof}

First we show that the quadratic equations in Set 3 and Set 4 can be deduced up to a $q_0$ power of a hyperplane from the equations in Set 1. Consider the quadratic equation $ef+ac+bd=0$ in Set 4, which can be given using the polygon in Figure \ref{figure1:46}.
\begin{figure}[htb!]
\centering
\begin{tikzpicture}
  [scale=1,auto=left,]
  \node [circle,draw] (C) at (1,1) {$C$};
  \node [circle,draw, above=1cm of C] (A)   {$A$};
  \node [circle,draw, right=1cm of C] (D)   {$D$};
  \node [circle,draw, right=1cm of A] (B)   {$B$};

  \draw [->] (A) -- node{$a$} (B);
  \draw [<-] (A) -- node[swap]{$d$} (C);
  \draw [->] (A) -- node[xshift=2mm]{$e$} (D);
  \draw [->] (B) -- node[swap,xshift=-1.8mm,yshift=-1.8mm]{$f$} (C);
  \draw [->] (B) -- node{$b$} (D);
  \draw [->] (C) -- node[swap]{$c$} (D);
\end{tikzpicture}
\caption{Polygon with four vertices.}\label{figure1:46}
\end{figure}
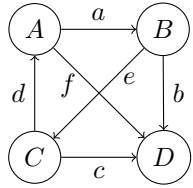

Now we fix the vertex $A$ and we consider the three triangles that contain the vertex $A$. Then, as explained above, we get the following three equations:
\begin{align*}
bA^{q_0}-fB^{q_0}+aD^{q_0}=0,\\
-cA^{q_0}+fC^{q_0}+dD^{q_0}=0,\\
eA^{q_0}+dB^{q_0}+aC^{q_0}=0.
\end{align*}
Next, we multiply the first equation by $d$ (note that $d$ is the label of the missing edge in the triangle $\Delta ABD$ inside the polygon), the second equation by $-a$, and the third by $f$ and we add them to get $A^{q_0}(ac+bd+ef)=0$.

Now we show that we can deduce the equations in Set 2 from the equations in Set 1. The strategy is as follows. We will show first that the equations in Set 1 define the Ree curve as it was described in \cite{Ped}, i.e., we will show that $w_1,\dots, w_{10}$ can be written in terms of $x,y_1,y_2$ as in \eqref{ReeEq:1}--\eqref{ReeEq:10}. Then, we will show that we get the two defining equations of the Ree curve \eqref{eq:Ree1} and \eqref{eq:Ree2}.

First we show $w_1=x^{3q_0+1} - y_1^{3q_0}$. Recall that we have the following equations in Sets 1 and 4.
\begin{align*}
&w_3^{q_0}+w_1-y_2x^{q_0}=0, \qquad \qquad \quad  &&\eqref{app:eq3}\\ 
&y_1x^{q_0}-w_2^{q_0}-y_2=0, \qquad \qquad \quad  &&\eqref{app:eq4}\\ 
&x^{q_0+1}-y_1-w_1^{q_0}=0,  \qquad \qquad \quad  &&\eqref{app:eq8}\\ 
&y_1^2-xy_2-w_4=0,          \qquad \qquad \quad   &&\eqref{app:Q6} \\
&xw_1=y_1y_2-v,             \qquad \qquad \quad   &&\eqref{app:Q12}\\ 
&y_1w_4=w_3+xw_2-xv.        \qquad \qquad \quad   &&\eqref{app:Q14}\\
\end{align*}
We multiply Equation \eqref{app:Q6} by $y_1$ to get
\begin{align*}
y_1^3&=xy_1y_2+y_1w_4\\
     &=xy_1y_2+w_3+xw_2-xv   &&\text{ by Equation } \eqref{app:Q14}\\
     &=x(y_1y_2-v)+w_3+xw_2\\
     &=x^2w_1+xw_2+w_3       &&\text{ by Equation } \eqref{app:Q12}\\
\end{align*}
Therefore, \begin{equation} \label{sec3:lemma**}
y_1^{3q_0}=x^{2q_0}w_1^{q_0}+x^{q_0}w_2^{q_0}+w_3^{q_0}
\end{equation}
Now we multiply Equation \eqref{app:eq8} by $x^{2q_0}$ to get
\begin{align*}
x^{3q_0+1}-x^{2q_0}y_1 -x^{2q_0}w_1^{q_0}&=0\\
x^{3q_0+1}-x^{2q_0}y_1+x^{q_0}w_2^{q_0}+w_3^{q_0}-y_1^{3q_0}&=0  && \text{ by } \eqref{sec3:lemma**}\\
x^{3q_0+1}-x^{2q_0}y_1+x^{q_0}w_2^{q_0}+y_2x^{q_0}-w_1-y_1^{3q_0}&=0 && \text{ by } \eqref{app:eq3}\\
x^{3q_0+1}-x^{q_0}(x^{q_0}y_1-w_2^{q_0}-y_2)-w_1-y_1^{3q_0} &=0 \\
x^{3q_0+1}-w_1-y_1^{3q_0} & = 0.&& \text{ by } \eqref{app:eq4}
\end{align*}
Therefore, we have $w_1=x^{3q_0+1}-y_1^{3q_0}$. Similarly, we can write $w_2,\dots,w_{10}$ in terms of $x,y_1,y_2$ to get the Equations \eqref{ReeEq:1}--\eqref{ReeEq:10} as in \cite{Ped}. Next, we use these equations to show that $y_1^{q}-y_1=x^{q_0}(x^q-x)$ and $y_2^q-y_2=x^{q_0}(y_1^q-y_1)$, but this has been done in the proof of Lemma \ref{sec3:lemma105}. Therefore, we have shown that the variety that is defined by the equations in Set 1 actually contains the Ree curve and hence we can get the equations in Set 2 up to some power of a hyperplane.
\end{proof}

We conclude this subsection by making a correspondence Pedersen's notations and Dickson's notation of Section \ref{sect2.1} which is given in Table \ref{table:D.1}.
\begin{table}[htb!]
\begin{center}
\begin{tabular}{ccccccc}
\toprule
$\xi_1$ & $\xi_2$ & $\xi_3$ & $\xi_0$ & $\mu_1$ & $\mu_2$ & $\mu_2$   \\
\midrule
$1$ & $-w_6$ & $-w_3$ & $w_2$ & $-w_8$ & $x$ & $w_1$\\
\bottomrule
\end{tabular}
\end{center}
\caption{Dickson's notation.}
\label{table:D.1}
\end{table}

The other important correspondence is between Pedersen's notations and Tits' notations \ref{sect2.1} which is given in Table \ref{table:T.1} and \ref{table:T.2}

\begin{table}[htb!]
\begin{center}
\begin{tabular}{ccccccc}
\toprule
$x_0$ & $x_1$ & $x_2$ & $x_{*}$ & $x_{0'}$ & $x_{1'}$ & $x_{2'}$   \\
\midrule
$-w_3$ & $-w_6$ & $1$ & $w_2$ & $w_1$ & $x$ & $-w_8$\\
\bottomrule
\end{tabular}
\end{center}
\caption{The coordinates $x_i$'s used by Tits and the corresponding rational functions notations.}
\label{table:T.1}
\end{table}

Moreover, we have that the set of $y_i$'s defined above corresponds to the Pedersen notation as in Table \ref{table:T.2}.
\begin{table}[htb!]
\begin{center}
\begin{tabular}{ccccccc}
\toprule
$y_0$ & $y_1$ & $y_2$ & $y_{*}$ & $y_{0'}$ & $y_{1'}$ & $y_{2'}$   \\
\midrule
$-w_5$ & $-y_{1}$ & $y_2$ & $w_7$ & $w_4$ & $w_{10}$ & $-w_{9}$\\
\bottomrule
\end{tabular}
\end{center}
\caption{The coordinates $y_i$'s used by Tits and the corresponding rational functions notations.}
\label{table:T.2}
\end{table}

\subsection{The Ree Group}\label{sec1:2.1}
After we have discussed the defining equations of the Ree curve, we discuss the construction of the Ree group as it appears in \cite{Wil2},\cite{Wil1}. In this subsection we will recall the new elementary construction of the Ree group $R(q)=\,\GG(q)$ as it is given in \cite[Chapter 4]{Wil2},\cite{Wil1}. The advantage of this approach is that it will avoid the use of Lie algebra. Let $V$ be a 7-dimensional $\fgq$-vector space with basis $\{i_t\,|\,t\in \fg{7}\}$, where $q:=3q_0^2:=3^{2m+1}\, (m \in \NN)$. Consider the anti-commutative multiplication on $V$ defined by $\cdot : V \times V \ni (i_t,i_{t+r})\mapsto i_{{t+3r}} \in V$ ($r=1,2,4$). This multiplication will define the 8-dimensional Octonian algebra $\OO$ with basis $\{1,i_t\,|\, t \in \fg{7}\}$ over $\fgq$. In fact $\OO$ is a Lie algebra with the Lie bracket is defined by the multiplication above. We consider two important maps $m:\Wedge(V)\ni i_t \land i_{t+r} \mapsto i_{t+3r} \in V $ ($r=1,2,4$) with kernel $W:=\ker(m)$ of dimension 14 over $\fgq$ and $\mu:V \ni i_t \mapsto \sum_{r=1,2,4}i_{t+r}\land i_{t+3r} \in \Wedge(V)$ with image $V':=\im(\mu) \simeq V$ of dimension 7 over $\fgq$. Moreover, we have that $V'=\im(\mu)\subseteq \ker(m)=W$.

Consider the $\fgq$-basis $\{ i'_t,i^*_t \,|\, t \in \fg{7} \}$ of $W$ defined by:
\begin{align*}
i'_t:=\mu(i_t)&=\sum_{r=1,2,4}i_{t+r}\land i_{t+3r}\\
              &= i_{t+1}\land i_{t+3}+i_{t+2}\land i_{t+6}+i_{t+4}\land i_{t+5},\\
i^*_t:&=i_{t+1}\land i_{t+3}-i_{t+2}\land i_{t+6}.
\end{align*}


To define the Ree group, we will need the following two homomorphisms.
\begin{alignat*}{2}
\mu = \theta : &V \to W                 \qquad \qquad \text{and} \qquad  \qquad   \rho : &&W \to V\\
               &i_t\mapsto i'_t                                                          &&i^*_t \mapsto i_t\\
               &\quad                                                                    && i'_t \mapsto 0.
\end{alignat*}

Note that $\rho$ induces an isomorphism between $V^*:=W/V'\simeq V$. The group $G_2(q)$ is then defined as the subgroup of the orthogonal group $\text{GO}_7(\fgq)$ (with orthonormal basis $\{i_t \,|\, t \in \fg{7} \}$ over $\fgq$) consisting of those elements which commute with $\theta$. This leads to an easy definition of the Ree group $\GG$ as follows. Consider the twisted map $\rho^*:W \to V$ given by $\rho^*(\lambda^*i^*_t)=\lambda i_t$, where $\lambda^*:=\lambda^{3q_0}$ ($\lambda\in \fgq$, i.e., $\lambda^{**}=\lambda^3$). Then, the Ree group $\GG(q)$ is defined as the subgroup of $G_2(q)$ consisting of those elements which commute with $\rho^*$.

Using the construction above, Wilson \cite{Wil1} gave a description of the $q^3+1$ $\fgq$-rational points, described the generators of the Ree group, and showed that $\GG$ is a simple group. To see this, change the basis of $V$ from $\{i_t \,|\, t \in \fg{7}\}$ to $\{v_{\pm 1}, v_{\pm 2},v_{\pm 3}, v_0 \}$ which is defined as follows
\begin{align*}
v_{-3}:=-i_3-i_5-i_6, &\qquad \qquad v_3:=i_3-i_5+i_6,\\
v_{-2}:=-i_1-i_2-i_4, &\qquad \qquad v_2:=i_1+i_2-i_4,\\
v_{-1}:=-i_0-i_3+i_6, &\qquad \qquad v_1:=-i_0+i_3-i_6,\\
			  &v_0:=i_1-i_2.
\end{align*}
Then, $W$ has the new basis $\{v^*_i,v'_i \,|\, i \in \{0,\pm 1, \pm 2, \pm 3 \} \}$ which is given in Table \ref{table:basis of W}.

\begin{table}[htb!]
\begin{center}
\begin{tabular}{  c c c  }
\toprule
$r$ & $v'_r $& $v_r^*$ \\
\midrule
-3 & $v_0 \land v_{-3} + v_{-2}\land v_{-1} $& $v_{-3} \land v_{-2}$\\
-2 & $v_1 \land v_{-3} + v_{-2}\land v_{0} $& $v_{-1} \land v_{-3}$\\
-1 & $v_{-3} \land v_{2} + v_{-1}\land v_{0} $& $v_{-2} \land v_{1}$\\

0 & $v_{3} \land v_{-3} + v_{2}\land v_{-2} + v_{1}\land v_{-1}$ & $v_{-3} \land v_{3} + v_{-2}\land v_{2}$\\
1 & $v_3 \land v_{-2} + v_{0}\land v_{1} $& $v_{2} \land v_{-1}$\\
2 & $v_{-1} \land v_{3} + v_{0}\land v_{2} $& $v_{1} \land v_{3}$\\
3 & $v_{3} \land v_{0} + v_{2}\land v_{1} $& $v_{3} \land v_{2}$\\
\bottomrule
 \end{tabular}

\end{center}
\caption{The new basis of $W$.}
\label{table:basis of W}
\end{table}


\begin{table}[htb!]
\begin{center}
\begin{tabular}{  c  c  c  c  c  c  c  c   }
\toprule
 & $v_{-3}$ & $v_{-2}$ & $v_{-1}$ & $v_{0}$ & $v_{1}$ & $v_{2}$ & $v_{3}$\\
\midrule
$v_{-3}$ & 0 & 0 & 0 & $-v_{-3}$ & $v_{-2}$ & $-v_{-1}$ & $v_{0}$\\
$v_{-2}$ &  & 0 & $v_{-3}$ & $v_{-2}$ & $0$ & $v_{0}$ & $v_{1}$\\
$v_{-1}$ &  &  & 0 & $v_{-1}$ & $-v_0$ & $0$ & $-v_{2}$\\
$v_{0}$ &  &  &  & 0 & $v_1$ & $v_2$ & $-v_{3}$\\
$v_{1}$ &  &  &  &  & 0 & $v_3$ & $0$\\
$v_{2}$ &  &  &  &  &  & 0 & $0$\\
$v_{3}$ &  &  &  &  &  &  & $0$\\
\bottomrule
 \end{tabular}

\end{center}
\caption{The multiplication table of $V$ using the new basis.}
\end{table}

A vector $v \in V$ is called a \emph{$*$-vector} if $v^*\equiv v \land w \pmod{V'}$, for some $w\in V$. Similarly $\langle v \rangle$ is called a \emph{$*$-point} if $v$ itself is a $*$-vector. Wilson \cite{Wil1} described explicitly the set of all $*$-points which is called the \emph{Ree unital} as follows. For any $*$-point $v \in V$, we have either $v=v_{-3}$ or $v=v_3+\sum_{r=-3}^{2}\alpha_r v_r$ ($\alpha_r \in \fgq$). Then, given $\alpha_2,\alpha_1,\alpha_0$, one can solve a system of equations described in \cite[Section 3]{Wil1} to find $v, w\in V$ such that $v^*\equiv v \land w \pmod{V'}$. Therefore, we have $q^3+1$ $*$-points in the Ree unital. Moreover, if $\phi$ is an automorphism of $\GG$ that fixes the point $\langle v_{-3} \rangle$, then $\phi$ is uniquely determined by $\alpha_2,\alpha_1,\alpha_0 \in \fgq$ in $\phi(v_3)=v_3+\sum_{r=-3}^{2}\alpha_r v_r$, and the system of equations in \cite[Section 3]{Wil1} provides enough information to solve for every entry of $\phi$. Note that the diagonal automorphism $\delta(\lambda):=\text{diag}(\lambda,\lambda^{3q_0-1}, \lambda^{-3q_0+2}, 1, \lambda^{3q_0-2}, \lambda^{-3q_0+1}, \lambda^{-1})$ is another automorphism that also fixes the $*$-point $\langle v_{-3} \rangle$. The subgroup $B\subseteq$ $\GG$ generated by these automorphisms is the maximal subgroup that fixes $\langle v_{-3} \rangle$. Moreover, the Ree group $\GG$ is generated by the subgroup $B$ and an automorphism of order 2.


\section{Smooth Embedding for the Ree Curve}\label{sec1:5}
In this section we want to prove that the variety $\mathcal{X}\subseteq \ppfgqc$ defined by the 105 equations of Section \ref{sec1:4} gives a smooth model for the Ree curve in the projective space. We denote the Ree curve by $\xr$ or simply by $\mathcal{R}$. Similar to the case of the Hermitian and Suzuki curves \cite{FTor1},\cite{TGK},\cite{Tor3}, we begin first by finding a very ample linear series that defines a smooth embedding for the Ree curve. We give then a concrete realization of the embedding by showing that the smooth curve $\mathcal{X}$ is birationally equivalent to the Ree curve. We recall that the $L$-polynomial of the Ree curve is given by
\[
 L(t):=L_{\xr}(t)=(qt^2+3q_0t+1)^a(qt^2+1)^b,
\]
where $a:=q_0(q^2-1)$, $b:=\sfrac{q_0(q-1)(q+3q_0+1)}{2}$ with $a+b=2\gr$.

We will follow the outline in the lecture notes \cite{Tor3}. The reciprocal polynomial of the $L$-polynomial is given by
\begin{align*}
h_{\xr}(t)&=t^{2g}L(t^{-1})=t^{2g}(qt^{-2}+3q_0t^{-1}+1)^a(qt^{-2}+1)^b\\
     &= t^{2g}\cdot t^{-2a}(q+3q_0t+t^2)^a \cdot t^{-2b}(q+t^2)^b\\
     &= (q+3q_0t+t^2)^a(q+t^2)^b.
\end{align*}
The polynomial $h_{\xr}$ has two irreducible factors $h_1(t):=q+3q_0t+t^2$ and $h_2(t):=q+t^2$. Set $h(t):=h_1(t)h_2(t)=q^2+3q_0qt+2qt^2+3q_0t^3+t^4$. Let $\Phi:\xr \to \xr$ be the Frobenius morphism on $\xr$ and $\JJr:=Cl^0(\xr)$ be the Jacobian group of $\xr$. Then, $\Phi$ induces a well-defined morphism $\tilde{\Phi}:\JJr \to \JJr$ given by $\tilde{\Phi}([P])=[\Phi(P)]$. Moreover, $h_{\xr}$ is the characteristic polynomial of $\tilde{\Phi}$ over $\fgqc$ \cite[Page 44]{Tor3}. Note that $\tilde{\Phi}$ is semisimple \cite[Chapter IV, Corollary 3]{Mum},\cite[Page 251]{Mum},\cite[Theorem 2 (a)]{Tate}. Therefore, $h(\tilde{\Phi})=0$ \cite[Chapter IV, Theorem 3]{Mum}, i.e., we have
\begin{equation} \label{eq:char}
 q^2I+3q_0q \tilde{\Phi} +2q \tilde{\Phi}^2 +3q_0 \tilde{\Phi}^3 + \tilde{\Phi}^4=0 \quad \text{ in } \JJr.
\end{equation}

Let $P_{\infty}\in \xr(\fgq)$ be as before the $\fgq$-rational point at infinity and let $f:\xr \to \JJr$ be the morphism defined by $f(P):=[P-P_\infty]$. Then, we have the following commutative diagram
\begin{center}
\begin{tikzpicture}
  [scale=1,auto=left,]
  \node [] (x1) at (1,1) {$J_R$};
  \node [] (xm2) at (1,4)   {$X_R$};
  \node [] (x2) at (4,1)  {$J_R$};
  \node [] (xm1) at (4,4)  {$X_R$};

  \draw [->] (xm2) -- node[swap]{$f$} (x1);
  \draw [<-] (x2) -- node{$\tilde{\Phi}$} (x1);

  \draw [->] (xm2) -- node{$\Phi$} (xm1);
  \draw [<-] (x2) -- node[swap]{$f$} (xm1);

\end{tikzpicture}
\end{center}
i.e., $f \circ \Phi= \tilde{\Phi} \circ f$.

\begin{lemma} \label{lemma:1}
For $P \in \xr$, we have
\begin{equation} \label{eq:equiv}
 q^2P+3q_0q \Phi(P) +2q \Phi^2(P) +3q_0 \Phi^3(P) + \Phi^4(P) \sim mP_\infty,
\end{equation}
where $m:=h(1)=q^2+3q_0q+2q+3q_0+1=-\nu_{\infty}(w_8)$.
\end{lemma}
\begin{proof}
First we notice that for all natural numbers $i=0,1,2,\dots$ and for any $P \in \xr$, $\tilde{\Phi}^i(f(P))=\tilde{\Phi}^{i}([P-P_\infty])=[\Phi^i(P)-\Phi^i(P_\infty)]$. Now we apply Equation \eqref{eq:char} to $f(P)\in \JJr$ to get that
\begin{alignat*}{2}
&q^2f(P)+3q_0q \tilde{\Phi}(f(P))+2q \tilde{\Phi}^2(f(P)) +3q_0 \tilde{\Phi}^3(f(P)) + \tilde{\Phi}^4(f(P)) = 0 \text{ in } J_{\text{R}}\\
&q[P-P_\infty]+3q_0q[\Phi(P)-\Phi(P_\infty)]+2q[\Phi^2(P)-\Phi^2(P_\infty)]\\
&  \qquad \qquad \qquad \qquad \,\, +3q_0[\Phi^3(P)-\Phi^3(P_\infty)]+[\Phi^4(P)-\Phi^4(P_\infty)] =0 \text{ in } J_{\text{R}}\\
&\left [ q^2P+3q_0q\Phi(P)+2q\Phi^2(P)+3q_0\Phi^3(P)+\Phi^4(P) \right ]=\\
&   \qquad  \qquad \qquad \left [ q^2P_\infty+3q_0q\Phi(P_\infty)+2q\Phi^2(P_\infty)+3q_0\Phi^3(P_\infty)+\Phi^4(P_\infty) \right ].
\end{alignat*}
Since $P_{\infty}\in \xr(\fgq)$, we get $\Phi(P_\infty)=P_\infty$. Therefore,
\[
 [ q^2P+3q_0q \Phi(P) +2q \Phi^2(P) +3q_0 \Phi^3(P) + \Phi^4(P)]=[(q^2+3q_0q+2q+3q_0+1)P_{\infty}]
\]
and we get the required equivalence
\[
  q^2P+3q_0q \Phi(P) +2q \Phi^2(P) +3q_0 \Phi^3(P) + \Phi^4(P) \sim mP_\infty.\qedhere
\]
\end{proof}

Fix $m:=h(1)=q^2+3q_0q+2q+3q_0+1$, $H:=mP_{\infty}$, $\Da:=\Da_{\text{R}}:=\left | (q^2+3q_0q+2q+3q_0+1)P_\infty \right | $,  $\Da':=\langle 1,x,y_1,y_2,w_1,\dots,w_{10} \rangle \subseteq \LL(H)$, and $\Da_1:=\mathbb{P}(\Da') \subseteq |mP_\infty|$. Note that we will show in Section \ref{sec1:8.1} that $\Da_1=\Da$ over $\fg{27}$. 

\begin{lemma}\label{lemma:2}
 With the notations above, we have the following:
\begin{enumerate}
 \item[(1)] $\Da$ is independent of the choice $P_\infty \in \xr(\fgq)$.

 \item[(2)] $m=q^2+3q_0q+2q+3q_0+1 \in \text{H}(Q)$, for all $Q\in \xr(\fgq)$, where $H(Q)$ is the Weierstrass non-gaps semigroup at $Q$.
 \item[(3)] $\Da$, $\Da_1$ are base-point-free and simple linear series.
 \item[(4)] $q^2$ is the first positive non-gap at every $Q\in \xr(\fgq)$.
\end{enumerate}
\end{lemma}
\begin{proof}
 (1) Let $Q\in \xr(\fgq)$ be another $\fgq$-rational point. Then, applying \eqref{eq:equiv} to $Q$, we get that $mQ \sim mP_{\infty}$ since $\Phi(Q)=Q$. Therefore, we have $\Da=\left | mP_{\infty}\right |=\left | mQ \right | $.


(2) Since $(w_8)_\infty=mP_{\infty}$, $m$ is a non-gap at $P_\infty$. By (1) we have $\left | mP_\infty\right | =\left | mQ \right | $ for any $\fgq$-rational point $Q\in \xr(\fgq)$. Therefore, there exists a positive divisor $A\in \left | mQ \right | $ such that $A=mQ+(z)=mP_\infty+(z')$, $z,z' \in \LL(H)$. Hence, $(z'z^{-1})_\infty=mQ$
 and so $m$ is a non-gap integer at $Q$.

(3) Recall that $\Da_1$ is a base-point-free if $b(P)=0$, for all points $P\in \xr$. Set $D_\infty:=mP_\infty+(w_8)\in \Da_1$. Therefore, we have $\nu_\infty(D_\infty)=m-m=0$, in particular, $b(P_\infty)=0$ and that shows $P_\infty$ is not a base point for $\Da_1$. For any other point $Q\ne P_\infty$, set $D_Q:=mP_{\infty}+(1) \in \Da_1$. Then, we have $\nu_Q(D_Q)=0$, in particular, $b(Q)=0$ and that shows $Q$ is not a base point for $\Da_1$. Therefore, $\Da$ is a base-point-free linear series and hence $\Da$ is also a base-point-free linear series.

To show $\Da_1$ is simple, we consider any morphism $\phi$ associated with $\Da_1$. We need to show that $\phi$ is birational, i.e., $\deg(\phi):=\left [\fgqc(\xr):\fgqc(\phi(\xr))\right ]=1$. Recall from Proposition 3.6.1 (c) in \cite{Sti} that $\deg(\phi)=\left [\fgqc(\xr):\fgqc(\phi(\xr))\right ]=\left [\fgq(\xr):\fgq(\phi(\xr))\right ]$.

 Now we consider first the morphism $\pi:=(t:x:y_1:y_2:w_1:\cdots:w_{10})$. Then, we have that $\deg((w_8)_\infty)=m=[\fgq(\xr):\fgq(w_8)]$ is divisible by $\deg(\pi)$, similarly $\deg((w_6)_\infty)=\deg((m-1)P_\infty)=m-1=[\fgq(\xr):\fgq(w_6)]$ is divisible by $\deg(\pi)$. Hence, we must have that $\deg(\pi)=1$. Now for any other morphism $\phi$ associated with $\Da_1$, there exists $\tau \in \Auti(\ppfgqc)$ such that $\phi=\tau \circ \pi$. Therefore, $\deg(\phi)=\deg(\pi)=1$. Therefore, $\Da_1$ is a simple linear series. Same argument also shows that $\Da$ is a simple linear series.

(4) Let $Q \in \xr(\fgq)$ be a rational point and let $n_1(Q)$ be the first non-gap integer at $Q$. We want to show $n_1(Q)=q^2$. Choose $P \in \xr$ such that $\Phi^i(P)\ne P$ ($i=0,1,2,3,4,5$) and $P$ is a non Weierstrass point according to the definition in \cite[p. 28]{Tor3}. Apply $\Phi$ to the equivalence \eqref{eq:equiv} and then subtract the result from \eqref{eq:equiv} we get
\[
 \Phi^5(P)+(3q_0-1)\Phi^4(P)+(2q-3q_0)\Phi^3(P)+(3q_0q-2q)\Phi^2(P) +(q^2-3q_0q)\Phi(P)\sim q^2 P.
\]
 Since the left-hand side of the equivalence above is a positive divisor and $P$ is not in its support, we have that the first non-gap at $P$ is less than or equal to $q^2$. By \cite[Lemma 2.30]{Tor3}, we have $n_1(Q) \leq n_1(P) \leq q^2$. Let $f \in \fr$ be the $\fgq$-rational function such that $(f)_\infty=n_1(Q)\cdot Q$. Then, $\fr \fieldextension \fgq(f)$ is of degree $n_1(Q)$. Now each $\fgq$-rational place of $\fgq(f)$ splits into at most $n_1(Q)$ rational places of $\fr$ with one specific place that will be totally ramified in $\fr \fieldextension \fgq(f)$ \cite[Theorem 1(b)]{Lewittes}. But then we have the bound $\#\xr(\fgq)=q^3+1\leq 1+qn_1(Q)$, but $n_1(Q)\leq q^2$, so $q^3+1\leq 1+qn_1(Q) \leq 1+q^3$. Therefore, $n_1(Q)=q^2$.
\end{proof}

\begin{remark} \label{remark:4}
\begin{enumerate}
 \item[(1)] For the point $P_\infty$ we have that $j_{N-i}(P_\infty)=m-n_{i}(P_\infty)$ $(i=0,1,2,\dots,N:=\ell(mP_\infty))$. Therefore, using Table \ref{table:4.1} we can determine some of the $(\Da,P_\infty)$-orders.
 \item[(2)] Since we have $h(t)=q^2+3q_0qt+2qt^2+3q_0t^3+t^4$, we get that $1,3q_0,2q,3q_0q,q^2$ are $\Da$-orders \cite[Corollary 4.22(1)]{Tor3}. Moreover, using Lemma \ref{lemma:2} and \cite[Corollary 4.22]{Tor3} we get that $\epsilon_{N}=\nu_{N-1}=q^2$, where $\epsilon_i$'s are the $\Da$-orders and $v_i$'s are the Frobenius orders of the linear series $\Da$ respectively.
 \item[(3)] \label{remark:4.3} Using Lemma 4.19 in \cite{Tor3}, we have $1,3q_0q,2q,3q_0$, and $q^2$ are $(\Da,P)$-orders for all $P\notin \xr(\fgq)$.
\end{enumerate}
\end{remark}

Next we show that the linear series $\Da$ is a very ample linear series. We will follow the proof of Proposition 8 in \cite{Suzz}.

\begin{prop}\label{prop1:proposition8}
The linear series $\Da=|mP_\infty|$ is a very amply linear series. 
\end{prop}
\begin{proof}
 Let $\varphi:\xr \to \mathbb{P}^{N-1}$ be the morphism associated to $\Da$. Since $\{1,x,y_1,y_2 \} \subseteq \LL(mP_\infty)$, $|mP_\infty|$ contains the linear series $\mathbb{P}(\langle 1,x,y_1,y_2\rangle)$ which induces a model for the Ree curve in $\mathbb{P}^{3}$ with a singularity at $P_\infty$. Thus, the morphism $\varphi$ is injective and separates tangent vectors (i.e., $\varphi$ has a non-zero differential) at any point in $\xr \setminus \{P_\infty\}$. In order to show $|mP_\infty|$ is a very ample linear series, we need only to show that $\ell (mP_\infty)=\ell ((m-2)P_\infty) +2$ \cite[Proposition 8]{Suzz}. But since $m$ and $m-1$ are non-gaps at $P_\infty$, then $\ell(mP_\infty)=\ell((m-1)P_\infty)+2$. Therefore, $|mP_\infty|$ is a very ample linear series. 
\end{proof}

Now we apply the idea of the proof of Proposition \ref{prop1:proposition8} to show that the linear series $\Da_1$ is a very ample of dimension 13.
\begin{theorem} \label{prop:5.5}
$\Da$ is a very ample linear series.
\end{theorem}

\begin{proof}


Recall that $\Da_1$ is a very ample linear series if $\Da_1$ separates points and tangent vectors \cite[Page 308]{Har}. To show $\Da$ separates points, let $P,Q\in \mathcal{R}$, we want to show that there exists a positive divisor $D \in \Da$ such that $P \in $ Supp($D$) and $Q \notin $ Supp($D$) (this is equivalent to showing that the morphism associated to $\Da$ is injective). From the equivalence \eqref{eq:equiv}, we have
\begin{align*}
 A:=&q^2P+3q_0q\Phi(P)+2q\Phi^2(P)+3q_0\Phi^3(P)+\Phi^4(P) \\
&\sim B:=q^2Q+3q_0q\Phi(Q)+2q\Phi^2(Q)+3q_0\Phi^3(Q)+\Phi^4(Q) \sim mP_\infty.
\end{align*}
If $Q\notin $ Supp($A$), then we are done. If $Q \in $ Supp($A$), then $Q=\Phi^i(P)$, for some $i=0,1,2,3,4$. Therefore, we have
\[
\{Q,\Phi(Q),\Phi^2(Q),\Phi^3(Q),\Phi^4(Q) \}= \{\Phi^{i}(P),\Phi^{i+1}(P),\Phi^{i+2}(P),\Phi^{i+3}(P),\Phi^{i+4}(P) \}.
 \]
Similarly, if we exchange the role of $P$ and $Q$, we have if $P\notin $ Supp($B$), then we are done. If $P \in $ Supp($B$), then $P=\Phi^j(Q)$ for some $j=0,1,2,3,4$. Therefore, we have $\{P,\Phi(P),\Phi^2(P),\Phi^3(P),\Phi^4(P) \}= \{\Phi^{j}(Q),\Phi^{j+1}(Q),\Phi^{j+2}(Q),\Phi^{j+3}(Q),\Phi^{j+4}(Q) \} $.
Therefore, $\Phi^{i+j}(P)=P$. Then, by examining the cases we get that $\Phi^5(P)=P$ or $\Phi^4(P)=P$ or $\Phi^3(P)=P$. But as in the Equation \eqref{eq:Reefqm} we have
\begin{align*}
 \#\xr(\fg{q^r})=q^r+1-2\sqrt{q^r}\left ( a\cos \left ( \frac{5\pi r}{6}\right )+b\cos \left ( \frac{\pi r}{6} \right )\right ).
\end{align*}
Hence, we have $\#\xr(\fgq)=q^3+1=\#\xr(\fg{q^2})=\#\xr(\fg{q^3})=\#\xr(\fg{q^4})=\#\xr(\fg{q^5})$. Therefore, we must have $P,Q\in \mathcal{R}(\fgq)$, and so $P=Q$. That means $\Da$ separates points and so $\Da_1$ separates points.

To show that $\Da_1$ separate tangent vectors, it is sufficient to show that $j^{\Da_1}_1(P)=1$ for all $P \in \mathcal{R}$ \cite{TGK}. For $P=P_\infty$, we have $\nu_{P_\infty}(mP_\infty+(w_6))=m-(m-1)=1$, so $j^{\Da_1}_1(P_\infty)=1$. Now we can apply the same idea of the proof in \cite[Proposition 8]{Suzz} as $P_\infty$ is the only singular point of the Ree curve model $\xr^{\text{sin}}$ defined by the two equations $y_1^{q}-y_1=x^{q_0}(x^q-x)$ and $y^{q}_2-y_2=x^{q_0}(y_1^{q}-y_1)$. So any $P \in \xr^{\text{sin}}\setminus \{ P_\infty \}$ separates points and tangent vectors. More concretely, for $P \in \xr(\fgq)\setminus \{P_\infty \}$, we need to show that $j_1(P)=1$ which is equivalent to finding a divisor $D=mP_\infty+(t_P) \in \Da_1$ with $\nu_P(t_P)=1$ and $t_{P}\in \langle t,x,y_1,y_2,w_1,\dots,w_{10} \rangle$.

Consider first the point $P_{000}$ and set $t_{000}:=x$, we have $\nu_{P_{000}}(t_{000})=e(P_{000}\placeextension P_{0})\cdot \nu_{P_{0}}(x)=1$ and clearly $x \in \Da'$. Therefore, $j_1(P_{000})=1$. Now we know that the maximal subgroup $G_{\fr}(P_\infty)$ that fixes the point $P_\infty$ acts linearly and transitively on the affine $\fgq$-rational points (Lemma \ref{lemma:7.1}) which means for any other place $P_{\alpha \beta \gamma}$, there exists $\psi \in G_{\fr}(P_{\infty})$ such that $\psi(P_{000})= P_{\alpha \beta \gamma}$, set $t_{\alpha \beta \gamma}:=\psi(t_{000})=\psi(x)$. Then, we have
\begin{align*}
 \nu_{P_{\alpha \beta \gamma}}(t_{\alpha \beta \gamma})&= \nu_{\psi(P_{000})}(t_{\alpha \beta \gamma})=\nu_{P_{000}}(\psi^{-1}(t_{\alpha \beta \gamma}))\\
       &= \nu_{P_{000}}(x)=1.
\end{align*}
Hence, $\nu_{P}(mP_{\infty}+(t_{\alpha \beta \gamma}))=1$ and so $j^{\Da_1}_1(P)=1$, for all $P \in \xr(\fgq)$.

Finally for a non-rational point $P \in \mathcal{R}$ with $P \notin \xr(\fgq)$ and $\Phi^i(P)\ne P$ $(i=0,1,2,3,4)$, we have $1,3q_0,2q,3q_0q,q^2$ are $(\Da,P)$-orders as in Remark \ref{remark:4} (3). Therefore $j^{\Da}_1(P)=1$ and so $j^{\Da_1}_1(P)=0,1$. But since for any $\fgq$-rational point $Q\in \xr(\fgq)$, $j^{\Da_1}_1(Q)=1$ and any point $P \notin \xr(\fgq)$ lies over some $Q\in \xr(\fgq)$ with ramification index one, we have that $j^{\Da_1}_1(P)=1$ as well. 
\end{proof}

\begin{cor}
The morphism $\pi=(t:x:y_1:y_2:w_1:\cdots:w_{10})$ associated to the very ample linear series $\Da=|mP_\infty|$ is a smooth embedding of the Ree curve in the projective space $\mathbb{P}^{13}$.
\end{cor}
\begin{proof}
 Lemma \ref{lemma:2} (3) and Theorem \ref{prop:5.5}.
\end{proof}

Now we will give a concrete realization of the Ree curve in $\ppfgqc$ using the morphism $\pi:=(1:x:y_1:y_2:w_1:\cdots:w_{10})$ as follows. We have from the Equations (5.3) in \cite{Tits} that
\begin{equation} \label{eq:3*}
\begin{aligned}
w_3&=x^{3q_0+3}-w_1^{3q_0}-x^2w_1-xw_2,\\
w_6&=x^{6q_0+3}-x^{3q_0}w_1^{3q_0}-w_2^{3q_0}-xw_1^2+w_1w_2,\\
w_8&=w_2^2-xw_6-w_1w_3.
\end{aligned}
\end{equation}
Moreover, we have from the 105 equations in Section \ref{sec1:4} the following equations:
\begin{equation}\label{eq:3**}
\begin{aligned}
y_1&=x^{q_0+1}-w_1^{q_0},\\
y_2&=x^{q_0}y_1-w_2^{q_0},\\
w_4&=xw_2^{q_0}-y_1w_1^{q_0},\\
w_7&=w_6^{q_0}-x^{q_0}w_4,\\
w_5&=w_8^{q_0}-x^{q_0}w_7,\\
w_9&=xw_8^{q_0}-w_1^{q_0}w_7,\\
w_{10}&=y_2w_6^{q_0}-w_3^{q_0}w_4.
\end{aligned}
\end{equation}

Therefore, the curve $\mathcal{R}$ can be given as the variety in the projective space $\ppfgqc$ defined by the set of points \[P_{\chi,\omega_1,\omega_2}:=\left [1:\chi:\upsilon_1:\upsilon_2:\omega_1:\cdots:\omega_{10}\right ]\]  and \[P'_{\infty}:=\pi(P_\infty)=\left [0:0:\cdots:0:1:0:0 \right ],\] where $\chi=x\in \fgqc, \omega_1=w_1 \in \fgqc$, and $\omega_2=w_2 \in \fgqc$ satisfying the equations \eqref{eq:w1x}, \eqref{eq:w2xw1}, and $\upsilon_1,\upsilon_2,\omega_i \in \fgqc$ are defined according to \eqref{eq:3*}, \eqref{eq:3**}. Using the results of Section \ref{sec1:6}, \cite{Ped},\cite{Tits} we will have that the Ree group acts as the automorphism group on the smooth model of the Ree curve.

\begin{remark}
Another way to show the variety $\mathcal{X}\subseteq \ppfgqc$ defined as the zero locus of the 105 equations (Sets 1--4) of Section \ref{sec1:4} is birationally equivalent to the Ree curve $\xr$ is to show directly that $\mathcal{X}$ is smooth curve with a function field isomorphic to the Ree function field.

First to show $\mathcal{X}$ is smooth, consider again the morphism $\pi:=(1:x:y_1:y_2:w_1:\cdots:w_{10})$ and let $P'_\infty:=\pi(P_\infty):=[0:\cdots:0:1:0:0]$. Using Equations \eqref{ReeEq:1}--\eqref{ReeEq:10}, the morphism $\pi$ is given by polynomial expressions in $1,x,y_1,y_2$. This implies that $\mathcal{X}$ is smooth at every affine point since the only singularity of the model $\xr^{\text{sin}}$ of the Ree curve is at $P_\infty$. To show $\mathcal{X}$ is smooth at $P'_\infty$, we have computed the Jacobian matrix of partial derivatives at $P'_\infty$ and we have found that it has the maximum possible rank which is 12. Therefore, $\mathcal{X}$ is a smooth curve in $\PP^{13}(\fgqc)$.

It remains to show that $\fgqc({\mathcal{X}})$ is isomorphic to the Ree function field, but this has been done in Lemma \ref{sec3:lemma105}. Therefore, $\mathcal{X}$ is a smooth model for the Ree curve in the projective space.

\end{remark}

\section{Relation to the Previous Work on the Embeddings of the Deligne-Lusztig Curves}\label{sec1:6}
In this section we relate the results of this paper with the work of Kane \cite{Kane}. Kane constructed smooth embeddings for the three Deligne-Lusztig curves associated to the groups $\AAA$ (Hermitian curve), $\BB$ (Suzuki curve), and $\GG$ (Ree curve). In the case of $G_2$, Kane gave an explicit (but rather abstract) embedding from $X_{\Gtwo} \to \PP(W)$, where $W$ is a representation of $\GG$ of dimension 14. For the curve constructed by Kane it is not clear if it is the Ree curve constructed by Pedersen \cite{Ped} or Tits \cite{Tits}, although both curves are of the same genus and having the same number of $\fgq$-rational points with $\GG$ as an automorphism group. Therefore, both curves are isomorphic by the uniqueness result of Hansen and Pedersen \cite{HP} (See Figure \ref{fig:E1}). In this section we will show that the set of $\fgq$-rational points for Kane's embedding \cite{Kane} is the same set of the $\fgq$-rational points of the embedding given in Section \ref{sec1:5}. This result depends on the construction of the Ree group $\GG$ given in Section \ref{sec1:2.1}. First we recall some notations from that section.

Let $V$ be a 7-dimensional vector space with basis $\{i_t: t \in \fg{7} \}$ and anti-commutative multiplication $i_t\cdot i_{t+r}:=i_{t+3r}$ $(r=1,2,4)$. Thus, $V$ can be identified with the pure imaginary part of the Octonion algebra $\OO$ \cite[Chapter 4]{Wil2}. Let $m:\Wedge(V) \to V$ be the map defined by the multiplication above and let $W:=\ker(m)$ which is a 14 dimensional vector space with basis $\{i_t^*,i_t'\}$ as defined in Section \ref{sec1:2.1}. Let $V':=\text{Im}(\mu) \subseteq \ker(m)=W$, where $\mu:V \to \Wedge(V)$ is defined by $\mu(i_t)=\sum_{r=1,2,4}i_{t+r}\land i_{t+3r}$. Set $V^*:=W/V'\simeq V$.

For $z \in \OO$ with $z=a+\sum_{t \in \mathbb{F}_{7}}a_ti_t$, we recall that the pure and imaginary part of $z$ are $\text{Re}(z):=a$ and $\text{Im}(z):=\sum_{t \in \fg{7}}a_ti_t$, respectively. Note that the map $m$ is the same map on the pure imaginary part of the Octonion algebra, i.e., $m$ is the map $*:V \times V \ni (x,y)\mapsto \text{Im}(x\cdot y)\in V$. Moreover, let $(-,-):V\times V \ni (x,y) \mapsto \text{Re}(x \cdot y) \in V$. Having these notations in place, we describe briefly the embedding given by Kane. For any given Borel subgroup $B \subseteq \Gtwo$, it fixes a complete flag
\[
0  \subseteq L=\langle x \rangle \subseteq M=\langle x,y\rangle \subseteq S=\langle x,y,z \rangle \subseteq S^{\perp} \subseteq M^{\perp} \subseteq L^{\perp} \subseteq V
\]
with the property $(M,M)=0=M*M$. Fix the canonical isomorphism $F:V^* \to V$. Then, $\sigma(B)$ is the Borel subgroup fixing $F(\langle x \land y \rangle)$ and $F(\langle x \land y,x \land z \rangle)$. We pick $w_{M,L} \in W(\Gtwo)$ such that for any $B \in X_{\Gtwo}(w_{M,L})$, we have that $B$ and $\sigma(B)$ are fixing the same line $L$ and plane $M$. We write $X_{\Gtwo}$ for $X_{\Gtwo}(w_{M,L})$. Then, we can define the embedding by
\begin{align*}
X_{\Gtwo} &\to \mathbb{P}(W)\\
B &\mapsto [w:=x \land y]
\end{align*}
Now the relations in \cite[Section 5.3]{Kane} that define the curve of this embedding guarantee that $w\in W$. Also, for any $w \in \wedge^2(V)$, if $ x=F(w)$, then $w=x \land y$, i.e., $x^* \equiv x \land y \pmod{V'}$. We get by \cite{Wil1} that the image of this embedding is the set of all $*$-points of the Ree group. Moreover, $B$ is the unique Borel subgroup such that $x \land y$ is parallel to $w$. Therefore, the embedding is well-defined and since the resulting smooth curve has the same genus, number of $\fgq$-rational points, and automorphism group $\GG(q)$ as the Ree curve, it is isomorphic to the Ree curve defined by Pedersen \cite{Ped,HP}. Note that $B \in X_{\Gtwo}$ corresponds to a $\fgq$-rational point in $\mathbb{P}^{13}$ if and only if $\sigma(B)=B$ if and only if $F(x \land y)$ is in the plane $M$.

To see that Kane's embedding is similar to our embedding at the level of $\fgq$-rational points, we use the new basis $\{v_i^*,v_i' \}$ of $W$ defined in Table \ref{table:basis of W}. First we have the correspondence between $1,x,y_1,\dots,w_{10}$ and the basis of $W$ as shown in Table~\ref{table1:WwithDa'}.
\begin{table}
\begin{alignat*}{3}
&v_{-3}^{'}=-w_{10},\qquad \qquad &&v_{-2}^{'}=-w_{9},\qquad \qquad &&v_{-1}^{'}=-w_{5},\\
&v_{3}^{'}=y_{1},\qquad \qquad &&v_{2}^{'}=y_{2},\qquad \qquad &&v_{1}^{'}=w_{4},\\
&v_{0}^{'}=w_7, \qquad \qquad &&v_{0}^{*}=w_2, \qquad            &&v_{-1}\land v_{1}=v\\
&v_{-3}^{*}=-w_{8},\qquad \qquad &&v_{-2}^{*}=x,\qquad \qquad &&v_{-1}^{*}=w_{1},\\
&v_{3}^{*}=1,\qquad \qquad &&v_{2}^{*}=-w_{6},\qquad \qquad &&v_{1}^{*}=-w_{3}.
\end{alignat*}
\caption{The correspondence between the basis of $W$ and the basis of $\Da'$.}\label{table1:WwithDa'}
\end{table}
Since the image of the embedding over $\fgq$ is the $*$-points $v \in V$ such that $v^* \equiv v \land w \pmod{V'}$, for some $w\in V$ \cite{Wil1}, the $\fgq$-rational points of Kane can be found by taking $\alpha_2,\alpha_1,\alpha_0\in \fgq$, say $v =v_3 + \sum_{r=-3}^{2}\alpha_r v_r$ and we find $\alpha_{-3},\alpha_{-2},\alpha_{-1}$ using Wilson's algorithm \cite[Section 3]{Wil1}. This will also give us $w=\sum_{r=-3}^{2}\beta_r v_r$ such that $v^* \equiv v \land w \pmod{V'}$. Write $v \land w$ in the basis $\{v_i^*,v_i' \}$ of $W$, i.e.,
\[
 v \land w =\sum_{r=-3}^{3}a_r v_r^* + \sum_{r=-3}^{3}b_r v'_r.
\]
Then, this will give a point in $\mathbb{P}(W)$ which corresponds to
\[
\left [ a_{3}:a_{-2}:b_{3}:b_{2}:a_{-1}:a_{0}:a_{1}:b_{1}:b_{-1}:a_{2}:b_{0}:a_{-3}:b_{-2}:b_{-3} \right ] \in  \ppfgqc
\]
by the isomorphism $\mathbb{P}(W) \ni [\sum_{i=1}^{14}{n_iw_i}] \mapsto [n_1:\cdots:n_{14}] \in \ppfgqc$. This point is a $\fgq$-rational point in our embedding because it is the evaluation of the rational functions $x,y_1,y_2$, $w_1,\dots,w_{10}$ at $x=\alpha_2,y_1=\alpha_1,y_2=\alpha_0$ on the Ree curve. Moreover, we have $v_{-3}^*=v_{-3}\land v_{-2}$ corresponds to the point at infinity $P'_{\infty}=[0:0:0: \cdots: 0 :1:0 :0]$. Note that we have used the fact that the Ree unital defined as the set of all $*$-rational points is the set $\Gamma$ of $\fgq$-rational points in the Ree curve.

\begin{figure}[htb!]
\centering

\begin{tikzpicture}
  [scale=1,auto=left, minimum size=5em]

  \node [rectangle,draw,text width=4cm, align = center] (b) at (1,1)  {Kane's embedding \cite{Kane}};
  \node [rectangle,draw,text width=4cm, align = center] (a) at (1,9)  {Ree curve as a Deligne-Lusztig curve,  \ref{sec1:2.2}};

  \node [rectangle,draw,text width=4cm, align = center] (c) at (5,5) {Uniqueness theorem of Hansen and Pedersen \cite{HP}};

  \node [rectangle,draw,text width=4cm, align = center] (e) at (9,1)  {Embedding in $\ppfgqc$, Section \ref{sec1:5}};
  \node [rectangle,draw,text width=4cm, align = center] (d) at (9,9)   {Ree curve defined by equations \cite{Ped}};


  \draw [->] (a) --    (c);
  \draw [->] (d) --   (c);
  \draw [->] (a) --  (b);
  \draw [->] (d) --  (e);
  \draw [->] (b) --  (e);

\end{tikzpicture}
\caption{The relation between the two embeddings.}\label{fig:E1}
\end{figure}
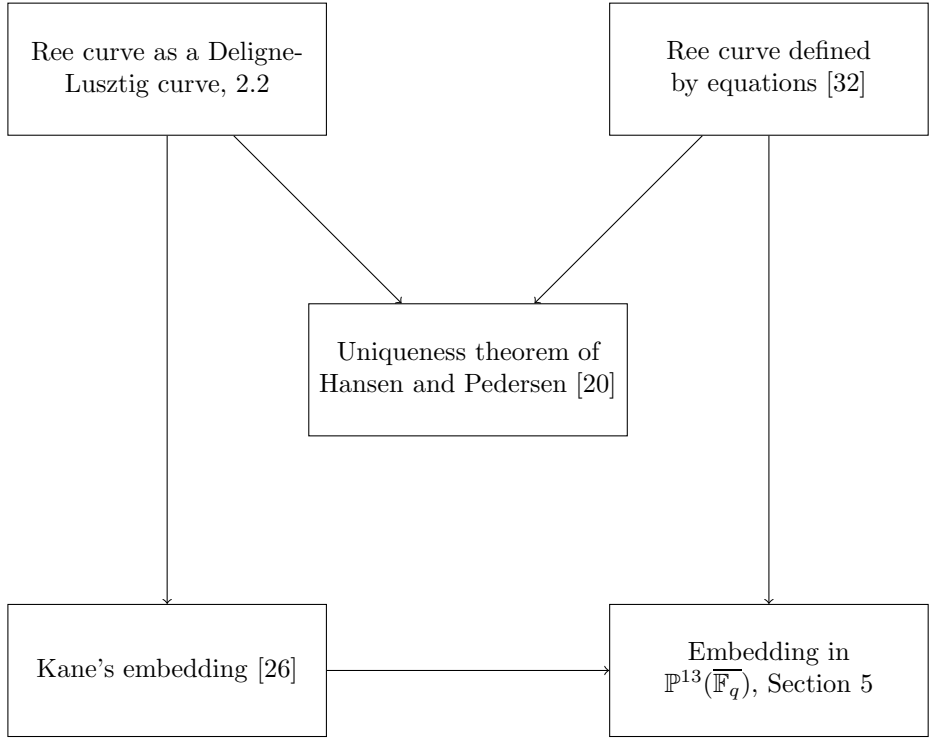




\section{Representation of the Ree Group}\label{sec1:7}
The Ree group $\GG(q)$ is the automorphism group of the Ree curve $\xr$ \cite{Ped}. In this section we show how the Ree group acts on the smooth model $\mathcal{R}$ of the Ree curve in the projective space $\ppfgqc$. First we show that the subgroup $G_{\fr}(P_\infty)$ of $\GG(q)$ which fixes the point at infinity $P_\infty$ acts linearly on the space $\Da'$ generated by $1,x,y_1,y_2,w_1,\dots,w_{10}$ and the action can be represented by lower triangular matrices.

\begin{lemma}\label{lemma:7.1}
 The subgroup $G_{\fr}(P_\infty)$ that fixes the point $P_{\infty}$ acts linearly on the space $\Da'=\langle 1,x,\dots,w_{10} \rangle$ over $\fgq$. Moreover, the action can be represented by lower triangular matrices.
\end{lemma}

\begin{proof} We recall from \cite{Ped} that \[G_{\fr}(P_\infty):=\{\psii \in \Auti(\fr) \,: \, \alpha \in \fgq^\times, \beta,\gamma,\delta \in \fgq\},\] where
\begin{align*}
 \psii(x)&:=\alpha x + \beta,\\
 \psii(y_1)&:= \alpha^{q_0+1}y_1 + \alpha \beta^{q_0}x + \gamma,\\
 \psii(y_2)&:=\alpha^{2q_0+1}y_2 - \alpha^{q_0+1}\beta^{q_0}y_1 +\alpha \beta^{2q_0} x + \delta.
\end{align*}
Direct calculations using the 105 equations of Section \ref{sec1:4} yield the action in the Appendix \ref{app:2}. This shows that $G_{\fr}(P_\infty)$ acts linearly on $\Da'$. Moreover, it is clear from the action above that it can be represented by lower triangular matrices.
\end{proof}

\begin{remark}\begin{enumerate}
  \item[(1)] From the proof of Lemma \ref{lemma:7.1}, we see that the subgroup $G_{\fr}(P_\infty)$ also acts linearly on the space $V:=\langle 1,x,w_1,w_2,w_3,w_6,w_8 \rangle$. This is the same action of the subgroup $G_{\fr}(P_\infty)$ on a 7-dimensional space as in \cite{Tits}.
  \item[(2)] $G_{\fr}(P_{\infty})$ is the subgroup $B$ used in Wilson \cite{Wil1}, see Section \ref{sec1:2.1}. 
 \end{enumerate}
\end{remark}

Next, we show that the Ree group acts on $\Da'$.

\begin{lemma}\label{lemma:7.2}
 The Ree group $\GG(q)$ acts on $\Da'$
\end{lemma}
\begin{proof}
We recall that the Ree group contains an involution map $\phi:\fr \to \fr$ that swap $x_{i} \longleftrightarrow x_{-i}$ and $y_{i} \longleftrightarrow y_{-i}$. As in \cite{Ped}, $\phi$ can be given explicitly by $\phi:\fr \to \fr$, where
\begin{alignat*}{3}
1 &\mapsto \frac{w_8}{w_8}, \quad x &\mapsto \frac{w_6}{w_8}, \quad y_1 &\mapsto \frac{w_{10}}{w_8},\\
y_2 &\mapsto \frac{w_9}{w_8}, \quad w_1 &\mapsto \frac{w_3}{w_8}, \quad w_4 &\mapsto \frac{w_5}{w_8},\\
w_2 &\mapsto \frac{w_2}{w_8}, \quad w_7 &\mapsto \frac{w_7}{w_8}, \quad v &\mapsto \frac{v}{w_8}.
\end{alignat*}
Therefore, apart from the $w_8$ factor, this automorphism maps $1 \longleftrightarrow w_8$, $x \longleftrightarrow w_6$, $w_1 \longleftrightarrow w_3$, $y_1 \longleftrightarrow w_{10}$, $y_2 \longleftrightarrow w_9$, $w_4 \longleftrightarrow w_5$. Moreover, $\phi$ sends the point at infinity $P_\infty$ to the point $P_{000}$. Since $\GG(q)$ is generated by $G_{\fr}(P_\infty)$ and $\phi$, the result follows and we have $\GG(q)$ acts on $\Da$.
\end{proof}
\begin{remark}
 \begin{enumerate}
 \item[(1)] From the proof of Lemma \ref{lemma:7.2}, we see that the Ree group $\GG(q)$ acts on the space $V:=\langle 1,x,w_1,w_2,w_3,w_6,w_8 \rangle$. This is the same action of $\GG(q)$ on a 7-dimensional space as in \cite{Tits}.

 \item[(2)] $\phi\circ \psi_{-1,0,0,0}$ is similar to the automorphism $\alpha^{\beta^3}$ defined in \cite{Wil1} which acts by negating $v_0$ and swapping $v_r$ with $v_{-r}$ ($r=1,2,3$).
\end{enumerate}
\end{remark}


\section{Further Properties of the Ree Curve}\label{sec1:8}
In this section we study three more properties of the Ree curve over specific finite fields. First we answer the 20 years old problem from \cite{Ped}, which is to find the Weierstrass non-gaps semigroup at $P_\infty$. We find that contrary to the case of the Hermitian and Suzuki curves, the pole orders of the basis functions $1,x,y_1$,$y_2,w_1,\dots,w_{10}$ of $\LL(mP_\infty)$ over $\fg{27}$ do not generate the Weierstrass non-gaps semigroup $H(P_\infty)$. 

\subsection{The Weierstrass Non-gaps Semigroup $H(P_\infty)$ over $\fg{27}$}\label{sec1:8.1}
The open problem from Pedersen's paper \cite{Ped} in 1993 is to find the set of all non-gaps at $P_\infty$. In this subsection we are interested in finding the Weierstrass non-gaps semigroup $H(P_\infty)$ of the Ree curve over $\fg{27}$. We show that the pole orders (Table \ref{table:4.1}) of the basis functions $1,x,y_1$,$y_2,w_1,\dots,w_{10}$ of $\LL(mP_\infty)$ that is used in Section \ref{sec1:5} to define the smooth embedding for the Ree curve in the projective space do not generate the full Weierstrass non-gaps semigroup $H(P_\infty)$. We mention that for the Hermitian and Suzuki curves, the pole orders of the basis functions of $\LL(mP_\infty)$ do generate the Weierstrass non-gaps semigroup $H(P_\infty)$.

\begin{table}[htb!]
\begin{center}
\begin{tabular}{  c c r c l  }
\toprule
$f$ &$~~~$ & $\nu_0(f)$ &$~~~$ &$\nu_{\infty}(f)$ \\
\midrule
$x$ & & $1 $ & & $-(q^2) $ \\
$y_1$ & & $q_0+1 $ & & $-(q^2+q_0q) $ \\
$y_2$ & & $2q_0+1 $ & & $-(q^2+2q_0q) $ \\
$w_1$ & & $3q_0+1 $ & & $-(q^2+3q_0q) $ \\
$w_2$ & & $q+3q_0+1 $ & & $-(q^2+3q_0q+q) $ \\
$w_3$ & & $2q+3q_0+1 $ & & $-(q^2+3q_0q+2q) $ \\
$w_4$ & & $q+2q_0+1 $ & & $-(q^2+2q_0q+q) $ \\
$v$ & & $2q+3q_0+1 $ & & $-(q^2+3q_0q+q) $ \\
$w_5$ & & $q_0q+q+3q_0+1 $ & & $-(q^2+3q_0q+q+q_0) $ \\
$w_6$ & & $3q_0q+2q+3q_0+1 $ & & $-(q^2+3q_0q+2q+3q_0) $ \\
$w_7$ & & $q_0q+q+2q_0+1 $ & & $-(q^2+2q_0q+q+q_0) $ \\
$w_8$ & & $q^2+3q_0q+2q+3q_0+1 $ & & $-(q^2+3q_0q+2q+3q_0+1) $ \\
$w_9$ & & $q_0q+2q+3q_0+1 $ & & $-(q^2+3q_0q+2q+q_0) $ \\
$w_{10}$ & & $2q_0q+2q+3q_0+1 $ & & $-(q^2+3q_0q+2q+2q_0) $ \\
\bottomrule
\end{tabular}
\end{center}
\caption{The zero and pole orders of the rational functions $x,y_1,y_2,w_1,\dots, w_{10}$ at $P_{000}$ and $P_{\infty}$.}
\label{table:4.1}
\end{table}

In particular, for the Hermitian curve, the functions $1,x,y$ generate the space $\LL(mP_\infty)$ with pole orders $0,q_0,q_0+1$ at $P_\infty$ respectively. Hence, we have $\langle q_0,q_0+1 \rangle\subseteq H(P_\infty)$, but $\NN \setminus \langle q_0,q_0+1 \rangle=\sfrac{q_0(q_0-1)}{2}=\gh$. Therefore, $H(P_\infty)=\langle q_0,q_0+1 \rangle$ by the Weierstrass gap theorem (\cite{Duursma1}). 

For the Suzuki curve, the functions $1,x,y,z,w$ generate the space $\LL(mP_\infty)$ with pole orders $0,q,q+q_0,q+2q_0,q+2q_0+1$ at $P_\infty$ respectively. Hence, we have $\langle q,\, q+q_0,\, q+2q_0,\, q+2q_0+1 \rangle\subseteq H(P_\infty)$, but $\NN \setminus \langle q,\, q+q_0,\, q+2q_0,\, q+2q_0+1\rangle=q_0(q-1)=\gs$ \cite{HS}. Therefore, $H(P_\infty)=\langle q,\,q+q_0,\,q+2q_0,\,q+2q_0+1 \rangle$. On the other hand, this does not hold for the Ree curve, more precisely, the pole orders of the basis functions $1,x,y_1,y_2,w_1,\dots,w_{10}$ of $\LL(mP_\infty)$ over $\fg{27}$ do not generate the full Weierstrass non-gaps semigroup $H(P_\infty)$. 

In establishing the set of nongaps we use the following observations about the Ree curve. First $P_\infty$ is a Weierstrass point since $q^2<\gr$ is the first non-gap integer at $P_\infty$ (Lemma \ref{lemma:2}). Second the divisor $(2\gr-2)P_\infty$ is a canonical divisor since $d(\sfrac{w_6}{w_8})=(2\gr-2)P_\infty$ \cite{Ped}. Third we have $(3q_0-2)\cdot m = 2\gr-2$. Therefore, over $\fg{27}$ we have $7H=7\cdot mP_\infty=(2\gr-2)P_\infty$ is a canonical divisor and so $-\nu_\infty(w_8^{3q_0-2})=(3q_0-2)m=2\gr-2$ is a non-gap integer. This implies that $2\gr-1$ is a gap \cite[Remark 4.4]{Pellikaan}. Which means that the Weierstrass non-gaps semigroup $H(P_\infty)$ is symmetric. 

Using the observations above, we have in the interval $0,1,\dots,2\gr-1=7253$, 3627 gaps and $3627$ non-gaps (since $\gr=3627$). A straightforward calculation of the semigroup generated by the pole orders of the functions $1,x,y_1,y_2,w_1,\dots,w_{10}$ will create only 3040 non-gap integers in the interval $0,1,\dots,2\gr-1=7253$ which proves the above mentioned claim. Next, we compute the full list of non-gap integers that are less than $2\gr$ as follows. 

Since we know that $H(P_\infty)$ is symmetric, whenever we find a non-gap integer $a$, we would know that $2\gr-1-a$ is a gap. Since we have the numbers $1,2,\dots,q^2-1=728$ are gap integers by Lemma \ref{lemma:2}, $(2\gr-1)-(1,2,\dots,728)$ are non-gap integers. Therefore, we will have 3100 non-gaps and 3100 gaps. Moreover, we use the fact that $7H$ is a canonical divisor and that the number of non-gaps in each equivalence class $a$ modulo $26$ is equal to $139$ if $a$ is odd and $140$ if $a$ is even. 

To compute the gaps and non-gaps in the interval $0,1,\dots,7253=2\gr-1$ for the Ree curve over $\fg{27}$, we used the computer algebra system MAGMA \cite{MAGMA}. For a function $f$ with poles only at $P_\infty$ we compute the pole order as the dimension of the ring $R/(I,f)$, where $R/I$  is the coordinate ring of the curve. The functions in $\LL((3q_0-2)mP_\infty)$ generate the affine ring of functions regular outside $P_\infty$. We use the defining equations of the curve to reduce the set of generating functions to a smaller generating set. In this set there are functions with the same pole order. We create new functions by taking differences of functions with the same pole order. The second round of reduction brings the set of generators in a form where all functions have different pole orders. At that point the non-gaps are known. In the second round we can divide the functions into groups of sizes $139$ or $140$ for pole orders in a given residue class modulo $26$. The semigroup of Weierstrass non-gaps is generated by the $132$ non-gaps in Table \ref{table:nongaps}. 

\begin{table}[htb!]
\begin{center}
\begin{tabular}{ rrrrrrrrrrrrrrr  }
\toprule
  729  &810  &891  &918  &921  &972  &999  &1002  &1026  &1029  &1032  &1035  &1036  \\
1866 &2520  &2547  &2601  &2604  &2628  &2631  &2658  &2706  &2709  &2712  &2739  &2820  \\
3250 &3277  &3285  &3286  &3287  &3312  &3313  &3314  &3331  &3358  &3366  &3367  &3368  \\
3393 &3394  &3395  &3396  &3444  &3447  &3471  &3474  &3477  &3498  &3501  &3504  &3507  \\
3557 &3558  &3584  &3585  &3592  &3612  &3619  &3638  &3665  &3673  &3700  &3703  &3750  \\
3751 &3754  &3777  &3778  &3781  &3784  &3804  &3805  &3808  &3811  &3814  &3862  &3863  \\
3865 &3889  &3890  &3892  &3899  &3919  &3926  &3943  &3944  &3946  &3947  &3970  &3971  \\
3973 &3974  &3980  &4000  &4001  &4007  &4010  &4047  &4048  &4049  &4051  &4052  &4054  \\
4055 &4057  &4058  &4061  &4081  &4082  &4084  &4085  &4088  &4091  &4111  &4112  &4115  \\
4118 &4121  &4174  &4201  &4228  &4237  &4238  &4240  &4241  &4481  &4484  &4508  &4511  \\
4535 &4538 \\
\bottomrule
\end{tabular}

\end{center}
\caption{The $132$ non-gaps that generate the Weierstrass semigroup for the Ree curve over $\fg{27}$ with genus $\gr = 3627$.}
\label{table:nongaps}
\end{table}



\appendix
\section{Equations of the Ree Curve}\label{app:1}
In this appendix we list the complete list of 105 equations that define the Ree cuve. These equations will be used in Sections \ref{sec1:4}, \ref{sec1:5}.

\begin{description}

\item[Set 1 and Set 2] 35 equations of degree $q_0+1$ of the form $aA^{q_0}+bB^{q_0}+cC^{q_0}=0$, where $A,B,C \in \{1,x,w_1,w_2,w_3,w_6,w_8 \}$ and $a,b,c \in \{1,x,\dots,w_{10}\}$ and 35 equations of degree $3q_0+1$ of the form $a^{3q_0}A+b^{3q_0}B+c^{3q_0}C=0$, where $a,b,c,A,B,C$ are the same functions in Set 1. Note that $v:=w_7-w_2$.

\begin{minipage}{.45\linewidth}
\begin{equation}\label{app:eq1}
x^{q_0}w_{4} + 2y_{1}w_{2}^{q_0} + y_{2}w_{1}^{q_0}= 0.
\end{equation}
\end{minipage}%
\begin{minipage}{0.55\linewidth}
\begin{equation}\label{app:eq1'}
xw_{4}^{3q_0} + 2y_{1}^{3q_0}w_{2} + y_{2}^{3q_0}w_{1}= 0.
\end{equation}
\end{minipage}

\begin{minipage}{.45\linewidth}
\begin{equation}\label{app:eq2}
2x^{q_0}y_{2} + w_{1} + w_{3}^{q_0}= 0.
\end{equation}
\end{minipage}%
\begin{minipage}{0.55\linewidth}
\begin{equation}\label{app:eq2'}
2xy_{2}^{3q_0} + w_{1}^{3q_0} + w_{3}= 0.
\end{equation}
\end{minipage}

\begin{minipage}{.45\linewidth}
\begin{equation}\label{app:eq3}
2w_{2}^{q_0}w_{6} + 2w_{3}^{q_0}w_{10} + 2w_{5}w_{8}^{q_0}= 0.
\end{equation}
\end{minipage}%
\begin{minipage}{0.55\linewidth}
\begin{equation}\label{app:eq3'}
2w_{2}w_{6}^{3q_0} + 2w_{3}w_{10}^{3q_0} + 2w_{5}^{3q_0}w_{8}= 0.
\end{equation}
\end{minipage}

\begin{minipage}{.45\linewidth}
\begin{equation}\label{app:eq4}
x^{q_0}y_{1} + 2y_{2} + 2w_{2}^{q_0}= 0.
\end{equation}
\end{minipage}%
\begin{minipage}{0.55\linewidth}
\begin{equation}\label{app:eq4'}
xy_{1}^{3q_0} + 2y_{2}^{3q_0} + 2w_{2}= 0.
\end{equation}
\end{minipage}

\begin{minipage}{.45\linewidth}
\begin{equation}\label{app:eq5}
w_{1}^{q_0}w_{10} + w_{2}^{q_0}w_{9} + 2w_{4}w_{8}^{q_0}= 0.
\end{equation}
\end{minipage}%
\begin{minipage}{0.55\linewidth}
\begin{equation}\label{app:eq5'}
w_{1}w_{10}^{3q_0} + w_{2}w_{9}^{3q_0} + 2w_{4}^{3q_0}w_{8}= 0.
\end{equation}
\end{minipage}

\begin{minipage}{.45\linewidth}
\begin{equation}\label{app:eq6}
2xw_{6}^{q_0} + w_{1}^{q_0}w_{4} + w_{3}= 0.
\end{equation}
\end{minipage}%
\begin{minipage}{0.55\linewidth}
\begin{equation}\label{app:eq6'}
2x^{3q_0}w_{6} + w_{1}w_{4}^{3q_0} + w_{3}^{3q_0}= 0.
\end{equation}
\end{minipage}

\begin{minipage}{.45\linewidth}
\begin{equation}\label{app:eq7}
2y_{2}w_{6}^{q_0} + w_{3}^{q_0}w_{4} + w_{10}= 0.
\end{equation}
\end{minipage}%
\begin{minipage}{0.55\linewidth}
\begin{equation}\label{app:eq7'}
2y_{2}^{3q_0}w_{6} + w_{3}w_{4}^{3q_0} + w_{10}^{3q_0}= 0.
\end{equation}
\end{minipage}

\begin{minipage}{.45\linewidth}
\begin{equation}\label{app:eq8}
x^{q_0+1} + 2y_{1} + 2w_{1}^{q_0}= 0.
\end{equation}
\end{minipage}%
\begin{minipage}{0.55\linewidth}
\begin{equation}\label{app:eq8'}
x^{3q_0+} + 2y_{1}^{3q_0} + 2w_{1}= 0.
\end{equation}
\end{minipage}

\begin{minipage}{.45\linewidth}
\begin{equation}\label{app:eq9}
x^{q_0}w_{10} + y_{2}w_{8}^{q_0} + w_{2}^{q_0}w_{5}= 0.
\end{equation}
\end{minipage}%
\begin{minipage}{0.55\linewidth}
\begin{equation}\label{app:eq9'}
xw_{10}^{3q_0} + y_{2}^{3q_0}w_{8} + w_{2}w_{5}^{3q_0}= 0.
\end{equation}
\end{minipage}

\begin{minipage}{.45\linewidth}
\begin{equation}\label{app:eq10}
2w_{1}^{q_0}w_{8} + 2w_{3}w_{8}^{q_0} + w_{6}^{q_0}w_{9}= 0.
\end{equation}
\end{minipage}%
\begin{minipage}{0.55\linewidth}
\begin{equation}\label{app:eq10'}
2w_{1}w_{8}^{3q_0} + 2w_{3}^{3q_0}w_{8} + w_{6}w_{9}^{3q_0}= 0.
\end{equation}
\end{minipage}

\begin{minipage}{.45\linewidth}
\begin{equation}\label{app:eq11}
y_{2}w_{8}^{q_0} + 2w_{3}^{q_0}w_7 + 2w_{6}= 0.
\end{equation}
\end{minipage}%
\begin{minipage}{0.55\linewidth}
\begin{equation}\label{app:eq11'}
y_{2}^{3q_0}w_{8} + 2w_{3}w_7^{3q_0} + 2w_{6}^{3q_0}= 0.
\end{equation}
\end{minipage}

\begin{minipage}{.45\linewidth}
\begin{equation}\label{app:eq12}
2x^{q_0}v + y_{1}w_{3}^{q_0} + 2w_{1}^{q_0}w_{1}= 0.
\end{equation}
\end{minipage}%
\begin{minipage}{0.55\linewidth}
\begin{equation}\label{app:eq12'}
2xv^{3q_0} + y_{1}^{3q_0}w_{3} + 2w_{1}w_{1}^{3q_0}= 0.
\end{equation}
\end{minipage}

\begin{minipage}{.45\linewidth}
\begin{equation}\label{app:eq13}
w_{1}^{q_0}w_{6} + 2w_{3}^{q_0}w_{9} + v w_{8}^{q_0}= 0.
\end{equation}
\end{minipage}%
\begin{minipage}{0.55\linewidth}
\begin{equation}\label{app:eq13'}
w_{1}w_{6}^{3q_0} + 2w_{3}w_{9}^{3q_0} + v^{3q_0}w_{8}= 0.
\end{equation}
\end{minipage}

\begin{minipage}{.45\linewidth}
\begin{equation}\label{app:eq14}
2w_{3}^{q_0}w_{8} + w_{6}^{q_0} w_{6} + 2w_{8}^{q_0}w_{10}= 0.
\end{equation}
\end{minipage}%
\begin{minipage}{0.55\linewidth}
\begin{equation}\label{app:eq14'}
2w_{3}w_{8}^{3q_0} + w_{6}w_{6}^{3q_0} + 2w_{8}w_{10}^{3q_0}= 0. 
\end{equation}
\end{minipage}

\begin{minipage}{.45\linewidth}
\begin{equation}\label{app:eq15}
2w_{4}w_{8}^{q_0} + w_{6}^{q_0}w_7 + w_{8}= 0.
\end{equation}
\end{minipage}%
\begin{minipage}{0.55\linewidth}
\begin{equation}\label{app:eq15'}
2w_{4}^{3q_0}w_{8} + w_{6}w_7^{3q_0} + w_{8}^{3q_0}= 0.
\end{equation}
\end{minipage}

\begin{minipage}{.45\linewidth}
\begin{equation}\label{app:eq16}
2w_{2}^{q_0}w_{8} + 2w_{6}^{q_0}w_{10} + w_{8}^{q_0}w_{9}= 0.
\end{equation}
\end{minipage}%
\begin{minipage}{0.55\linewidth}
\begin{equation}\label{app:eq16'}
2w_{2}w_{8}^{3q_0} + 2w_{6}w_{10}^{3q_0} + w_{8}w_{9}^{3q_0}= 0.
\end{equation}
\end{minipage}

\begin{minipage}{.45\linewidth}
\begin{equation}\label{app:eq17}
xw_{3}^{q_0} + 2y_{2}w_{1}^{q_0} + 2v= 0.
\end{equation}
\end{minipage}%
\begin{minipage}{0.55\linewidth}
\begin{equation}\label{app:eq17'}
x^{3q_0}w_{3} + 2y_{2}^{3q_0}w_{1} + 2v^{3q_0}= 0.
\end{equation}
\end{minipage}

\begin{minipage}{.45\linewidth}
\begin{equation}\label{app:eq18}
2y_{1}w_{8}^{q_0} + w_{2}^{q_0}w_7 + 2w_{10}= 0.
\end{equation}
\end{minipage}%
\begin{minipage}{0.55\linewidth}
\begin{equation}\label{app:eq18'}
2y_{1}^{3q_0}w_{8} + w_{2}w_7^{3q_0} + 2w_{10}^{3q_0}= 0.
\end{equation}
\end{minipage}

\begin{minipage}{.45\linewidth}
\begin{equation}\label{app:eq19}
2xw_{8}^{q_0} + w_{1}^{q_0}w_7 + w_{9}= 0.
\end{equation}
\end{minipage}%
\begin{minipage}{0.55\linewidth}
\begin{equation}\label{app:eq19'}
2x^{3q_0}w_{8} + w_{1}w_7^{3q_0} + w_{9}^{3q_0}= 0.
\end{equation}
\end{minipage}

\begin{minipage}{.45\linewidth}
\begin{equation}\label{app:eq20}
x^{q_0}w_{5} + 2y_{2}w_{3}^{q_0} + w_{1}w_{2}^{q_0}= 0.
\end{equation}
\end{minipage}%
\begin{minipage}{0.55\linewidth}
\begin{equation}\label{app:eq20'}
xw_{5}^{3q_0} + 2y_{2}^{3q_0}w_{3} + w_{1}^{3q_0}w_{2}= 0.
\end{equation}
\end{minipage}

\begin{minipage}{.45\linewidth}
\begin{equation}\label{app:eq21}
2x^{q_0}w_{9} + y_{1}w_{8}^{q_0} + w_{1}^{q_0}w_{5}= 0.
\end{equation}
\end{minipage}%
\begin{minipage}{0.55\linewidth}
\begin{equation}\label{app:eq21'}
2xw_{9}^{3q_0} + y_{1}^{3q_0}w_{8} + w_{1}w_{5}^{3q_0}= 0.
\end{equation}
\end{minipage}

\begin{minipage}{.45\linewidth}
\begin{equation}\label{app:eq22}
2x^{q_0}w_7 + 2w_{5} + w_{8}^{q_0}= 0.
\end{equation}
\end{minipage}%
\begin{minipage}{0.55\linewidth}
\begin{equation}\label{app:eq22'}
2x w_7^{3q_0} + 2w_{5}^{3q_0} + w_{8}= 0.
\end{equation}
\end{minipage}

\begin{minipage}{.45\linewidth}
\begin{equation}\label{app:eq23}
x^{q_0}w_{6} + 2w_{1}w_{8}^{q_0} + 2w_{3}^{q_0}w_{5}= 0.
\end{equation}
\end{minipage}%
\begin{minipage}{0.55\linewidth}
\begin{equation}\label{app:eq23'}
xw_{6}^{3q_0} + 2w_{1}^{3q_0}w_{8} + 2w_{3}w_{5}^{3q_0}= 0.
\end{equation}
\end{minipage}

\begin{minipage}{.45\linewidth}
\begin{equation}\label{app:eq24}
 2y_{1}w_{6}^{q_0} + w_{2}^{q_0}w_{4} + 2w_{9}= 0.
\end{equation}
\end{minipage}%
\begin{minipage}{0.55\linewidth}
\begin{equation}\label{app:eq24'}
2y_{1}^{3q_0}w_{6} + w_{2}w_{4}^{3q_0} + 2w_{9}^{3q_0}= 0.  
\end{equation}
\end{minipage}

\begin{minipage}{.45\linewidth}
\begin{equation}\label{app:eq25}
w_{1}^{q_0}w_{9} + w_{2}^{q_0}w_{3} + 2w_{4}w_{6}^{q_0}= 0.
\end{equation}
\end{minipage}%
\begin{minipage}{0.55\linewidth}
\begin{equation}\label{app:eq25'}
w_{1}w_{9}^{3q_0} + w_{2}w_{3}^{3q_0} + 2w_{4}^{3q_0}w_{6}= 0.
\end{equation}
\end{minipage}

\begin{minipage}{.45\linewidth}
\begin{equation}\label{app:eq26}
 2w_{1}^{q_0}w_{5} + w_{2}^{q_0}v + 2w_{3}^{q_0}w_{4}= 0.
\end{equation}
\end{minipage}%
\begin{minipage}{0.55\linewidth}
\begin{equation}\label{app:eq26'}
2w_{1}w_{5}^{3q_0} + w_{2}v^{3q_0} + 2w_{3}w_{4}^{3q_0}= 0.
\end{equation}
\end{minipage}

\begin{minipage}{.45\linewidth}
\begin{equation}\label{app:eq27}
w_{1}^{q_0}w_{10} + 2w_{3}^{q_0} w_{3} + vw_{6}^{q_0}= 0.
\end{equation}
\end{minipage}%
\begin{minipage}{0.55\linewidth}
\begin{equation}\label{app:eq27'}
w_{1}w_{10}^{3q_0} + 2w_{3}w_{3}^{3q_0}+ v^{3q_0}w_{6}= 0.
\end{equation}
\end{minipage}

\begin{minipage}{.45\linewidth}
\begin{equation}\label{app:eq28}
w_{2}^{q_0}w_{10} + w_{3}^{q_0}w_{9} + w_{5}w_{6}^{q_0}= 0.
\end{equation}
\end{minipage}%
\begin{minipage}{0.55\linewidth}
\begin{equation}\label{app:eq28'}
w_{2}w_{10}^{3q_0} + w_{3}w_{9}^{3q_0} + w_{5}^{3q_0}w_{6}= 0.
\end{equation}
\end{minipage}

\begin{minipage}{.45\linewidth}
\begin{equation}\label{app:eq29}
2y_{1}w_{3}^{q_0} + y_{2}w_{2}^{q_0} + w_{5}= 0. 
\end{equation}
\end{minipage}%
\begin{minipage}{0.55\linewidth}
\begin{equation}\label{app:eq29'}
2y_{1}^{3q_0}w_{3} + y_{2}^{3q_0}w_{2} + w_{5}^{3q_0}= 0.
\end{equation}
\end{minipage}

\begin{minipage}{.45\linewidth}
\begin{equation}\label{app:eq30}
xw_{2}^{q_0} + 2y_{1}w_{1}^{q_0} + 2w_{4}= 0.
\end{equation}
\end{minipage}%
\begin{minipage}{0.55\linewidth}
\begin{equation}\label{app:eq30'}
x^{3q_0}w_{2} + 2y_{1}^{3q_0}w_{1} + 2w_{4}^{3q_0}= 0.
\end{equation}
\end{minipage}

\begin{minipage}{.45\linewidth}
\begin{equation}\label{app:eq31}
x^{q_0}w_{9} + y_{2}w_{6}^{q_0} + w_{2}^{q_0}(w_{2}+w_7)= 0. 
\end{equation}
\end{minipage}%
\begin{minipage}{0.55\linewidth}
\begin{equation}\label{app:eq31'}
xw_{9}^{3q_0} + y_{2}^{3q_0}w_{6} + w_{2}(w_{2}w_7)^{3q_0}= 0.
\end{equation}
\end{minipage}

\begin{minipage}{.45\linewidth}
\begin{equation}\label{app:eq32}
x^{q_0}w_{8} + 2w_{5}w_{6}^{q_0} + w_{8}^{q_0}(w_{2}+w_7)= 0.
\end{equation}
\end{minipage}%
\begin{minipage}{0.55\linewidth}
\begin{equation}\label{app:eq32'}
xw_{8}^{3q_0} + 2w_{5}^{3q_0}w_{6} + w_{8}(w_{2}+w_7)^{3q_0}= 0. 
\end{equation}
\end{minipage}

\begin{minipage}{.45\linewidth}
\begin{equation}\label{app:eq33}
2x^{q_0}w_{3} + y_{1}w_{6}^{q_0} + w_{1}^{q_0}(w_{2}+w_7)= 0.
\end{equation}
\end{minipage}%
\begin{minipage}{0.55\linewidth}
\begin{equation}\label{app:eq33'}
2xw_{3}^{3q_0} + y_{1}^{3q_0}w_{6} + w_{1}(w_{2}+w_7)^{3q_0}= 0.  
\end{equation}
\end{minipage}

\begin{minipage}{.45\linewidth}
\begin{equation}\label{app:eq34}
2x^{q_0}w_{10} + w_{1}w_{6}^{q_0} + w_{3}^{q_0}(w_{2}+w_7)= 0.
\end{equation}
\end{minipage}%
\begin{minipage}{0.55\linewidth}
\begin{equation}\label{app:eq34'}
2xw_{10}^{3q_0} + w_{1}^{3q_0}w_{6} + w_{3}(w_{2}+w_7)^{3q_0}= 0.
\end{equation}
\end{minipage}

\begin{minipage}{.45\linewidth}
\begin{equation}\label{app:eq35}
x^{q_0}w_{4} + 2w_{6}^{q_0} + (w_{2}+w_7)= 0.
\end{equation}
\end{minipage}%
\begin{minipage}{0.55\linewidth}
\begin{equation}\label{app:eq35'}
xw_{4}^{3q_0} + 2w_{6} + (w_{2}+w_7)^{3q_0}= 0.
\end{equation}
\end{minipage}

\item[Set 3] The equation
\begin{equation}\label{app:eqset3}
-w_2^2+w_8+xw_6+w_1w_3=0.
\end{equation}

\item[Set 4] 34 quadratic equations of the form $f_{R_{ab}}f_{R_{cd}}+f_{R_{ad}}f_{R_{bc}}+f_{R_{ac}}f_{R_{db}}=0$, where $f_{R_{ab}}$ is the function such that $f_{R_{ab}}^{3q_0} \sim R_{ab}$ in the Table \ref{table:R}


\begin{align}
&2y_{1}w_{8} + w_{2}w_{9} + 2w_{4}w_{10} + vw_{9}= 0.\label{app:Q1} \\ 
&w_{5}w_{8} + w_{6}w_{9} + 2w_{10}^2= 0.\label{app:Q2} \\ 
&2y_{1}w_{1} + y_{2}^2 + 2w_{5}= 0.\label{app:Q3} \\ 
&2xw_{8} + 2w_{2}w_{3} + 2w_{3}v + w_{4}w_{9}= 0.\label{app:Q4} \\ 
&w_{3}w_{5} + w_{4}w_{10} + vw_{9}= 0.\label{app:Q5} \\ 
&2xy_{2} + y_{1}^2 + 2w_{4}= 0.\label{app:Q6} \\ 
&xw_{10} + y_{1}w_{9} + 2w_{2}w_{4} + 2w_{4}v= 0.\label{app:Q7}  \\
&2xw_{5} + y_{1}v + 2y_{2}w_{4}= 0.\label{app:Q8} \\ 
&2w_{1}w_{8} + w_{2}w_{6} + 2vw_{6} + w_{5}w_{10}= 0.\label{app:Q9} \\ 
&2w_{3}w_{6} + vw_{8} + w_{9}w_{10}= 0.\label{app:Q10} \\ 
&2xw_{6} + y_{2}w_{9} + 2w_{2}v + 2v^2= 0.\label{app:Q11} \\ 
&2xw_{1} + y_{1}y_{2} + 2v= 0.\label{app:Q12} 
\end{align}
\begin{align}
&2y_{2}w_{10} + 2w_{1}w_{9} + 2w_{2}w_{5} + vw_{5}= 0.\label{app:Q13} \\ 
& 2xw_{2} + xv + y_{1}w_{4} + 2w_{3}= 0.\label{app:Q14}\\
& 2y_{1}w_{8} + w_{2}w_{9} + w_{3}w_{5} + 2vw_{9}= 0.\label{app:Q15} \\ 
&y_{2}w_{5} + w_{1}w_{2} + w_{1}v + 2w_{6}= 0.\label{app:Q16} \\ 
&2y_{1}w_{6} + 2y_{2}w_{10} + 2w_{2}w_{5} + 2vw_{5}= 0.\label{app:Q18} \\
&2y_{1}w_{10} + w_{1}w_{3} + 2w_{2}v + v^2= 0.\label{app:Q19} \\ 
&2y_{2}w_{2} + y_{2}v + w_{1}w_{4} + 2w_{10}= 0.\label{app:Q20} \\ 
&y_{1}w_{5} + y_{2}w_{2} + y_{2}v + w_{10}= 0.\label{app:Q21} \\ 
&2y_{2}w_{8} + 2w_{2}w_{10} + vw_{10} + 2w_{5}w_{9}= 0.\label{app:Q22} \\ 
&2y_{2}w_{6} + 2w_{1}w_{10} + w_{5}^2= 0.\label{app:Q23} \\ 
&w_{2}^2 + w_{4}w_{5} + 2v^2 + 2w_{8}= 0.\label{app:Q24} \\ 
&2xw_{10} + y_{2}w_{3} + 2w_{4}v= 0.\label{app:Q25} \\ 
&2y_{1}w_{10} + 2y_{2}w_{9} + 2w_{4}w_{5}= 0.\label{app:Q26} \\ 
&2y_{1}w_{6} + w_{1}w_{9} + vw_{5}= 0.\label{app:Q27} \\ 
&y_{1}w_{9} + y_{2}w_{3} + 2w_{2}w_{4} + w_{4}v= 0.\label{app:Q28} \\ 
&2y_{1}w_{2} + y_{1}v + y_{2}w_{4} + w_{9}= 0.\label{app:Q29} \\ 
&xw_{5} + y_{1}w_{2} + y_{1}v + 2w_{9}= 0.\label{app:Q30} \\ 
&w_{3}w_{10} + w_{4}w_{8} + 2w_{9}^2= 0.\label{app:Q31} \\ 
&w_{4}w_{6} + vw_{10} + w_{5}w_{9}= 0.\label{app:Q32} \\ 
&2y_{1}w_{5} + y_{2}v + 2w_{1}w_{4}= 0.\label{app:Q33} \\ 
&2y_{2}w_{8} + 2w_{2}w_{10} + w_{4}w_{6} + 2vw_{10}= 0.\label{app:Q34} \\ 
&xw_{9} + y_{1}w_{3} + 2w_{4}^2=0.\label{app:Q35} 
\end{align}

\end{description}


\section{The Action of the Group $G_{\fr}(P_\infty)$ on $\Da'$ }\label{app:2}

In this appendix we complete Lemma \ref{lemma:7.1}. We list how the group $G_{\fr}(P_\infty)$ acts on the set $\{w_1,\dots,w_{10}\}$. Let $\phi:=\phi_{\alpha,\beta,\gamma,\delta} \in G_{\fr}(P_\infty)$.

\begin{align*}
\psi(w_{1}):&=\alpha^{3q_0+1}w_{1}+\alpha \beta^{3q_0}x+\left (\beta^{3q_0+1}-\gamma^{3q_0} \right).\\
\psi(w_{2}):&=\alpha^{3q_0+2}w_{2}+\alpha^{3q_0+1}\beta w_{1}+\alpha \gamma^{3q_0}x\\
             &\quad \left (\beta \gamma^{3q_0}-\delta^{3q_0}\right ).\\
\psi(w_{3}):&=\alpha^{3q_0+3}w_{3}-\alpha^{3q_0+2}\beta w_{2}+\alpha^{3q_0+1}\beta^2w_{1}+a\delta^{3q_0}x\\
             &\quad \left (\beta\delta^{3*q_0}-\beta^{3q_0+3}+\gamma^3\right ).\\
\psi(w_{4}):&=\alpha^{2q_0+2}w_{4}-\alpha^{2q_0+1}\beta y_2+\left ( \alpha^{q_0+1}\beta^{q_0+1}-\alpha^{q_0+1}\gamma \right ) y_1\\
             &\quad \left( -\alpha \delta -\alpha \beta^{2q_0+1}-\alpha \beta^{q_0}\gamma  \right ) x+\left ( \gamma^2-\beta \delta\right ).
\end{align*}

\begin{align*}
\psi(w_{5})&=\alpha^{4q_0+2}w_{5}+\alpha^{3q_0+2}\beta^{q_0}v+\alpha^{2q_0+2}\beta^{2q_0}w_{4}-\alpha^{3q_0+1}\gamma w_{1}\\
            &\quad -\alpha^{2q_0+1}\delta y_{2}+\left( \alpha^{q_0+1}\beta^{q_0}\delta-\alpha^{q_0+1}\beta^{3q_0+1}+\alpha^{q_0+1}\gamma^{3q_0} \right(      	y_{1}\\
            &\quad \left ( -\alpha \beta^{2q_0}\delta -\alpha \beta^{4q_0+1}+\alpha \beta^{q_0}\gamma^{3q_0} -\alpha \beta^{3q_0}\gamma \right ) x\\
            &\quad \left(+\delta^2-\beta^{3q_0+1}\gamma+\gamma^{3q_0+1} \right ).\\
\psi(v) &= \alpha^{3 q_{0}+2} v - \alpha^{3 q_{0}+1} \beta w_{1} - \alpha^{2 q_{0}+2} \beta^{1 q_{0} } w_{4}  + \alpha^{2 q_{0}+1} \gamma y_{2}\\
&\quad - \alpha^{1 q_{0}+1} \beta^{1 q_{0} } \gamma y_{1} + \alpha^{1 q_{0}+1} \delta y_{1} - 2 \alpha \beta^{3 q_{0}+1} x + \alpha \beta^{2 q_{0} } \gamma x\\
&\quad + \alpha \beta^{1 q_{0} } \delta x + \alpha \gamma^{3 q_{0} } x - \beta^{3 q_{0}+2} + \beta \gamma^{3 q_{0} } + \gamma \delta.\\ 
\psi(w_6)&= \alpha^{6 q_{0}+3} w_{6}   - \alpha^{3 q_{0}+3} \beta^{3 q_{0} } w_{3}  + \alpha^{3 q_{0}+2} \beta^{3 q_{0}+1} w_{2}   - \alpha^{3 q_{0}+2} \gamma^{3 q_{0} } w_{2}  -    \alpha^{3 q_{0}+1} \beta^{3 q_{0}+2} w_{1}\\
           &\quad + 2 \alpha^{3 q_{0}+1} \beta \gamma^{3 q_{0} } w_{1} - \alpha^{3 q_{0}+1} \delta^{3 q_{0} } w_{1}     - \alpha \beta^{3 q_{0} } \delta^{3 q_{0} } x -    2 \alpha \gamma^{6 q_{0} } x - \beta^{6 q_{0}+3}\\
           &\quad - \beta^{3 q_{0}+1} \delta^{3 q_{0} } - 2 \beta \gamma^{6 q_{0} } + \gamma^{3 q_{0} } \delta^{3 q_{0} } + \delta^3.\\
\psi(w_7)&=\psi(w_2)+\psi(v).\\
\psi(w_8)&= \alpha^{6 q_{0}+4} w_{8} -\alpha^{6 q_{0}+3} \beta w_{6} +\alpha^{3 q_{0}+3} \gamma^{3 q_{0} } w_{3}+ 2 \alpha^{3 q_{0}+2} \beta \gamma^{3  q_{0} } w_{2}  - 2 \alpha^{3 q_{0}+2} \delta^{3 q_{0} } w_{2}\\
          &\quad+\alpha^{3 q_{0}+1} \beta^{3 q_{0}+3} w_{1} + \alpha^{3 q_{0}+1} \beta^2 \gamma^{3 q_{0} } w_{1} - 2 \alpha^{3 q_{0}+1} \beta \delta^{3 q_{0} } w_{1} -\alpha^{3 q_{0}+1} \gamma^3 w_{1} + 5 \alpha \beta^{6 q_{0}+3} x \\
          &\quad- \alpha \beta^{3 q_{0} } \gamma^3 x - 2 \alpha \gamma^{3 q_{0} } \delta^{3 q_{0} } x - \alpha \delta^3 x + 2 \beta^{6 q_{0}+4} - 4 \beta^{3 q_{0}+3} \gamma^{3 q_{0} } - \beta^{3 q_{0}+1} \gamma^3 + 3 \beta^2 \gamma^{6 q_{0} }\\
          &\quad - 2 \beta \gamma^{3 q_{0} } \delta^{3 q_{0} } - \beta \delta^3 + \gamma^{3 q_{0}+3} + \delta^{6 q_{0} }.\\
\psi(w_9)&=\alpha^{4 q_{0}+3} w_{9}  + \alpha^{4 q_{0}+2} \beta w_{5} - \alpha^{3 q_{0}+3} \beta^{1 q_{0} } w_{3}  + \alpha^{3 q_{0}+2} \beta^{1 q_{0}+1} v  + \alpha^{3 q_{0}+2} \gamma w_{2}\\
          &\quad - \alpha^{3 q_{0}+2} \gamma v - \alpha^{3 q_{0}+1} \beta \gamma w_{1} + \alpha^{2 q_{0}+2} \beta^{2 q_{0}+1} w_{4}  + \alpha^{2 q_{0}+2} \beta^{1 q_{0} } \gamma w_{4} + \alpha^{2 q_{0}+2} \delta w_{4}\\
          &\quad - \alpha^{2 q_{0}+1} \beta \delta y_{2} + \alpha^{2 q_{0}+1} \gamma^2 y_{2} - 2 \alpha^{1 q_{0}+1} \beta^{3 q_{0}+2} y_{1} + \alpha^{1 q_{0}+1} \beta^{1 q_{0}+1} \delta y_{1} - \alpha^{1 q_{0}+1} \beta^{1 q_{0} } \gamma^2 y_{1}\\
          &\quad + 3 \alpha^{1 q_{0}+1} \beta \gamma^{3 q_{0} } y_{1} + 2 \alpha^{1 q_{0}+1} \gamma \delta y_{1} - \alpha^{1 q_{0}+1} \delta^{3 q_{0} } y_{1}  - 2 \alpha \beta^{4 q_{0}+2} x - 4 \alpha \beta^{3 q_{0}+1} \gamma x\\
          &\quad - \alpha \beta^{2 q_{0}+1} \delta x + \alpha \beta^{2 q_{0} } \gamma^2 x + 3 \alpha \beta^{1 q_{0}+1} \gamma^{3 q_{0} } x +
           2 \alpha \beta^{1 q_{0} } \gamma \delta x - \alpha \beta^{1 q_{0} } \delta^{3 q_{0} } x\\
          &\quad + 3 \alpha \gamma^{3 q_{0}+1} x + \alpha \delta^2 x - 2 \beta^{3 q_{0}+2} \gamma + 3 \beta \gamma^{3 q_{0}+1}
          + \beta \delta^2 + \gamma^2 \delta - \gamma \delta^{3 q_{0} }.\\
\psi(w_{10})&=\alpha^{5 q_{0}+3} w_{10}  + \alpha^{4 q_{0}+3} \beta^{1 q_{0} } w_{9} + \alpha^{4 q_{0}+2} \gamma w_{5} + \alpha^{3 q_{0}+3} \beta^{2 q_{0} } w_{3}  + \alpha^{3 q_{0}+2} \beta^{1 q_{0} } \gamma v\\
            &\quad - \alpha^{3 q_{0}+2} \delta w_{2} - \alpha^{3 q_{0}+2} \delta v + \alpha^{3 q_{0}+1} \gamma^2 w_{1} +
    \alpha^{2 q_{0}+2} \beta^{3 q_{0}+1} w_{4}   + \alpha^{2 q_{0}+2} \beta^{2 q_{0} } \gamma w_{4}\\
            &\quad + \alpha^{2 q_{0}+2} \beta^{1 q_{0} } \delta w_{4} - \alpha^{2 q_{0}+2} \gamma^{3 q_{0} } w_{4} + \alpha^{2 q_{0}+1} \beta^{3 q_{0}+2} y_{2} - 2 \alpha^{2 q_{0}+1} \beta \gamma^{3 q_{0} } y_{2} - \alpha^{2 q_{0}+1} \gamma \delta y_{2}\\
            &\quad + \alpha^{2 q_{0}+1} \delta^{3 q_{0} } y_{2} - \alpha^{1 q_{0}+1} \beta^{4 q_{0}+2} y_{1} + 2 \alpha^{1 q_{0}+1} \beta^{3 q_{0}+1} \gamma y_{1} +    2 \alpha^{1 q_{0}+1} \beta^{1 q_{0}+1} \gamma^{3 q_{0} } y_{1} + \alpha^{1 q_{0}+1} \beta^{1 q_{0} } \gamma \delta y_{1}\\
            &\quad - \alpha^{1 q_{0}+1} \beta^{1 q_{0} } \delta^{3 q_{0} } y_{1} - 2 \alpha^{1 q_{0}+1} \gamma^{3 q_{0}+1} y_{1} - 2 \alpha^{1 q_{0}+1} \delta^2 y_{1} + \alpha \beta^{5 q_{0}+2} x + 2 \alpha \beta^{4 q_{0}+1} \gamma x\\
            &\quad + 2 \alpha \beta^{3 q_{0}+1} \delta x +\alpha \beta^{3 q_{0} } \gamma^2 x - 2 \alpha \beta^{2 q_{0}+1} \gamma^{3 q_{0} } x - \alpha \beta^{2 q_{0} } \gamma \delta x + \alpha \beta^{2 q_{0} } \delta^{3 q_{0} } x\\
            &\quad -    2 \alpha \beta^{1 q_{0} } \gamma^{3 q_{0}+1} x - 2 \alpha \beta^{1 q_{0} } \delta^2 x - 2 \alpha \gamma^{3 q_{0} } \delta x + \beta^{3 q_{0}+2} \delta + \beta^{3 q_{0}+1} \gamma^2\\
            &\quad - 2 \beta \gamma^{3 q_{0} } \delta - \gamma^{3 q_{0}+2} - 2 \gamma \delta^2 + \delta^{3 q_{0}+1}.
\end{align*}


\bibliographystyle{amsplain}

\newpage

\bibliography{mypaper,mypaper2,mypaper3,mypaper0}

\end{document}